\documentclass[preprint,12pt]{elsarticle}
\usepackage{amsmath,amssymb}
\usepackage{algorithm}
\usepackage{algorithmic}
\usepackage{lineno,hyperref}
\usepackage{graphicx}
\usepackage{multirow}
\usepackage{soul}
\usepackage{epstopdf}
\usepackage{dsfont}
\usepackage{bm}
\usepackage{color}

\biboptions{numbers,sort&compress}

\makeatletter
\newtheorem{theorem}{Theorem}[section]
\newtheorem{lemma}[theorem]{Lemma}

\newtheorem{corollary}{Corollary}[section]

\newtheorem{remark}{Remark}[section]

\newenvironment{proof}[1][Proof]{\noindent\textbf{#1.} }{\hfill $\Box$}
\numberwithin{equation}{section}

\allowdisplaybreaks[4]
\begin{document}

\begin{frontmatter}

\title{Spatial Decay and Asymptotic Stability  of Forced Waves at the Critical Speed for the Fisher-KPP Equation in a Shifting Environment}
\author[DANWEI1]{Zhibao Tang}\ead{tangzhibao1022@163.com}
\author[DANWEI2]{Lina Wang}\ead{wanglina@btbu.edu.cn}
\author[DANWEI1]{Shi-Liang Wu \corref{cor}}\ead{slwu@xidian.edu.cn}
\cortext[cor]{Corresponding author.}
\address[DANWEI1]{School of Mathematics and Statistics, Xidian University, Xi'an, Shaanxi 710071, P.R. China}
\address[DANWEI2]{School of Mathematics and Statistics, Beijing Technology and Business University
Beijing 100048, China}

\begin{abstract}
In this paper, we investigate the spatial decay and asymptotic stability of forced waves for the Fisher-KPP equation in a shifting heterogeneous environment, focusing on the critical speed $c_*=2\sqrt{r_+}$. Within the framework of regular variation, we establish a complete classification of critical wave tails in terms of the asymptotic behavior of the environmental deviation $r_+-r(z)$ as $z\to+\infty$. In each case, we derive precise spatial decay that characterize how the environmental tail modifies the leading-order decay of the wave at criticality. We further establish asymptotic stability estimates of forced waves. For initial data satisfying an exponentially weighted boundedness condition, we prove an $O(t^{-1/2})$ convergence rate in the uniform norm. Under stronger spatial decay assumptions on the weighted initial data, this rate improves to $O(t^{-1-\sigma/2})$, for $\sigma\in[0,1]$, in both the uniform norm and a weighted norm. Moreover, no smallness assumption on the initial perturbations is needed for our results.

\medskip
%{\it 2000 Mathematics Subject Classification}:
\end{abstract}
\begin{keyword}
Shifting environment; Fisher-KPP equation; traveling waves; spatial decay; asymptotic stability.

MSC:~35K57; 35C07; 35K15; 35B35; 35B40.
\end{keyword}

\end{frontmatter}

%% \linenumbers

\section{Introduction}
\label{}
\noindent
 Reaction-diffusion equations of the Fisher-KPP type are some of the main ways to investigate the
 spatial spread of biological populations, advantageous genes and invasive species.  Since the pioneering works of Fisher
\cite{Fisher1937} and Kolmogorov, Petrovskii and Piskunov
\cite{KPP1937}, a key question has been how local growth and spatial dispersal affect each other.
 The classical reaction-diffusion equation in a homogeneous environment
\begin{equation}\label{eq:homogeneous-kpp}
  u_t=u_{xx}+u(\alpha-u), \qquad x\in \mathbb{R},
\end{equation}
with $\alpha>0$, admits monotone traveling fronts
$u(t,x)=\phi(x-ct)$ connecting $\phi(-\infty)=\alpha>0$ to the
unstable state $\phi(+\infty)=0$ iff $c\geq 2\sqrt{\alpha}$.  The number
$2\sqrt{\alpha}$ is the classical KPP minimal wave speed and is determined by the
linearization at $u=0$.  The existence, qualitative properties,
and long-time behavior of these fronts have been widely studied, with early examples including
 \cite{AW1978,FifeMcLeod1977,Uchiyama1978}.
The leading-order asymptotics already reveal an important difference between
noncritical and critical fronts.  If $c>2\sqrt{\alpha}$, the classical front
has a purely exponential leading order term, whereas the critical front with
$c=2\sqrt{\alpha}$ contains an additional algebraic factor, of the form
$z{\rm e}^{-\sqrt{\alpha}z}$ up to multiplication by a positive constant.  Thus,
even in the homogeneous problem, the coincidence of the two characteristic
roots at the minimal speed changes the precise spatial decay.

Many natural environments often shift over time.  Climate
change, shifting boundaries, and the good resources moving from one place to another,
then the equations use a moving frame $x-ct$ to describe how growth changes.
In this case, $c$ is determined by the environment, rather than chosen by the population, and it is natural to ask whether the
population can persist, spread, or follow the moving habitat.
Some early and important work on shifting boundaries and climate-driven population dynamics was
studied by \cite{PL2004,Be2009}.  A
general theory for heterogeneous reaction-diffusion equations with a forced speed was extensively
investigated by generalized principal eigenvalues, stationary
solutions in the moving frame, and comparison methods
\cite{BR2008,Be2009}.  Persistence, extinction, and
spreading in various shifting-habitat have been studied, for example,
\cite{LBSF2014,GYZZ2026,BG2019,HSL2020,LamYu2022,LNY2025,Vo2015};
see \cite{WangEtAl2022} for a recent survey and further references.

Another particularly important phenomenon in a shifting environment is a
\emph{forced wave}: a population front whose speed agrees exactly with the
speed of the environmental shift.  Such waves are stationary in the shifting
environment and describe how the population shape follows the shifting environment.
Unlike ordinary traveling waves in a homogeneous equation, forced waves do not arise from
 translation invariance, and their existence, multiplicity,
and spatial decay may depend on the spatial heterogeneity.  For a
general class of asymptotically KPP nonlinearities,
Berestycki and Fang \cite{BF2018} provided a nearly complete theory on the existence, multiplicity, and attractivity of these waves
except for some degenerate cases.
They further showed that in heterogeneous environment, the leading order term determines the type of wave behavior.
Related questions
of uniqueness and stability were proved in
\cite{GGS2024,CTW2017,YiZhao2023}, and time-periodic shifting environments
were studied in \cite{FPZ2021,HWZ2025}.

In the present paper, we consider the spatial decay and asymptotic stability of forced waves for the following Fisher-KPP equation
\begin{equation}\label{initialvpbm}
 u_t=u_{xx}+u(r(x-ct)-u),~~x\in \mathbb{R},~t>0,
\end{equation}
where $c>0$ is the forcing speed and $r(z)(z=x-ct)$ is a spatially heterogeneous growth rate.
 We assume that
$r(\cdot)\in C(\mathbb{R})$, which satisfies
\begin{enumerate}
\item[(H1)]  $0<r_-\leq r(z)<r_+
  \,\,\,\text{for }z\in\mathbb{R}$~\text{with}~$r_-:=\lim\limits_{z\to-\infty}r(z)$~\text{and}~
  $r_+:=\lim\limits_{z\to+\infty}r(z).$
\end{enumerate}
Thus, the environment is favorable for the population at every location,
and it transitions from one positive growth rate on the left
to the higher one on the right.
As shown in \cite{BF2018}, there exist infinitely many ordered forced waves of \eqref{initialvpbm} of the form $u(t,x)=\phi(z)$ with $z=x-ct$, where its profile satisfies
\begin{equation}\label{1.3}
\phi''(z)+c\phi'(z)+\phi(z)\big{(}r(z)-\phi(z)\big{)}=0,\quad\xi\in \mathbb{R},
\end{equation}
subject to
\begin{equation}\label{in0}
\phi(-\infty)=r_->0~~{\rm and}~~\phi(+\infty)=0.
\end{equation}
%Looking for an exponential term ${\rm e}^{-\lambda z}$ gives the characteristic equation
%$
%  \lambda^2-c\lambda+r_+=0.
%$
%It follows that the critical speed and the two noncritical exponents are
%\begin{equation*}
%  c_*:=2\sqrt{r_+},
%  \qquad
%  \lambda_\pm(c):=\frac{c\pm\sqrt{c^2-4r_+}}{2}
%  \quad (c>c_*).
%\end{equation*}
%We have two cases to consider. In the noncritical case $c>c_*$, the two roots are distinct, with
%$0<\lambda_-(c)<\lambda_+(c)$, and the critical case $c=c_*$, in which the two roots coincide at
%$
%  \lambda_-(c)=\lambda_+(c)=\sqrt{r_+}.
%$
%In the homogeneous Fisher--KPP equation, the nonlinear front with
%\(c>c_*\) selects the slower decay rate $\lambda_-(c)$, while the critical front
%has the sharper asymptotic form $z{\rm e}^{-\lambda_*z}$.
In an asymptotically
homogeneous shifting environment, the selection of spatial decay is more complex,
which depends on both the wave speed $c$ and the sizes of $r_-$ and $r_+$.
Berestycki and Fang \cite{BF2018} showed that, under $r_->r_+$,
only the minimal wave is selected with the fast exponential decay rate,
while all the other waves tend to zero with the slow exponential rate
in the noncritical case $c>c_*:=2\sqrt{r_+}$; however,
all waves decay only at the slow exponential rate when $r_-<r_+$.
They also investigated
the critical speed $c=c_*$ and recovered the classical critical tail under some assumption  on the spatial heterogeneity.  More precisely, for the
general equation $u_t=u_{xx}+f(x-ct,u)$, under an exponentially fast convergence
condition satisfying
\begin{equation}\label{eq:BF}
  \sup_{0<s\leq s_0}
  \left|\frac{f(z,s)-f_\infty(s)}{s}\right|
  =o\bigl({\rm e}^{-\eta z}\bigr),
  \quad\text{as }z\to+\infty,
  \qquad f_\infty(s):=\lim_{z\to+\infty}f(z,s),
\end{equation}
for some \(s_0,\eta>0\), they obtained for  $c=c_*$
\begin{equation}\label{BF-critical-tail}
  \phi(z)=A z{\rm e}^{-\frac{c}{2}z}+o\bigl(z{\rm e}^{-\frac{c}{2}z}\bigr),
  \quad\text{as }z\to+\infty
\end{equation}
with $A>0$.  For the logistic case, condition \eqref{eq:BF} says that
$r(z)$ converges to $r_+$ exponentially,
Berestycki and Fang \cite{BF2018} proved that the forced waves decay like \eqref{BF-critical-tail}.

Recently, this exponential decay condition \eqref{eq:BF} was relaxed by Tang et.al \cite{TangWuXu}.
In that work, the authors studied the existence, stability, and precise spatial decay of
forced waves for \eqref{initialvpbm} at
both noncritical and critical forcing speeds.  For the model \eqref{1.3},
at the critical speed it is enough to assume
\begin{equation}\label{first-order}
  \int_{z_0}^{+\infty} z\,|r(z)-r_+|\,{\rm d} z<+\infty,
\end{equation}
which replaces \eqref{eq:BF}.
This also has the classical asymptotic behavior satisfying \eqref{BF-critical-tail}.
As a condition on the tail of the environment,
\eqref{first-order} is weaker than exponential
convergence; for instance, it allows algebraic decay such as
$|r(z)-r_+|=O(z^{-p})$ with $p>2$.  This result only works for fast-decaying rate.
It does not cover other cases where the decay is too slow or critical.
The present paper deals with the cases left out above, and gives a complete classification of the resulting
critical tails.  In particular, $z{\rm e}^{-cz/2}$ appears as the integrable case.

A similar but different problem appears when the growth rate at infinity is degenerate, that is, when $r_+=0$.
For \eqref{1.3} with $r_+=0$ in such a degenerate shifting environment,
Tang et.al \cite{TangWuWu2026,TWW2026JDE} established existence and nonexistence, multiplicity, and
the precise spatial decay of all forced waves as well as the stability, including both exponentially
and non-exponentially decaying cases.  In \cite{TangWuWu2026}, the slow decay behavior is
caused by a characteristic root tending to zero.  Instead, the present
paper assumes $r_+>0$, and the new tails are generated by the coupling of a
double nonzero root at $c_*=2\sqrt{r_+}$.  Its asymptotic mechanism is different from the critical
non-degenerate classification developed here.

Furthermore, algebraic stability at the critical speed is already an established
feature of the homogeneous Fisher--KPP problem. In particular, Faye and
Holzer~\cite{FH2019} give a pointwise Green-function proof of
weighted $t^{-3/2}$ stability for sufficiently small localized
perturbations of the homogeneous critical front.
To the best of my knowledge, there are few results on the asymptotic Fisher-KPP equation with a moving environment.

The purpose of this paper is to give a complete asymptotic
classification of the tails of all forced waves near $z=+\infty$ satisfying
\eqref{1.3}-\eqref{in0} at the critical speed $c=c_*$.
We prove how the exact decay is affected by the asymptotic behavior of
$q(z)=r_+-r(z)$ as $z\to+\infty$, covering the integrable case \eqref{first-order}, critical, and non-integrable cases
considered below.  In particular, the classical profile
$A z{\rm e}^{-cz/2}$ is recovered under the condition
\eqref{first-order}, whereas different algebraic or
sub-exponential corrections arise when that condition fails.
We are more interested in the asymptotic stability of all forced waves .  In particular,
we do not need smallness of initial datum, but we do retain the global lower bound by
a positive multiple of the forced wave.

%Therefore, the role of regular variation in our analysis is also important. First, it provides a general description of the convergence of $r(z)$ toward $r_+$, substantially extending the pure power-law regimes. Second, it produces a natural classification of the tail behavior through the index $p$, with $p=2$ regarding as the critical threshold. Third, it provides the asymptotic information, which is needed to construct and estimate the Green kernel associated with the linearized problem.

\subsection{Motivation and main results}

Our first main question is that how the decay rate and the accumulated effect of $r_+-r(z)$ influence the decay of the forced-wave $\phi(z)$ as $z\to+\infty$.
Throughout the paper, we consider the critical speed $c=2\sqrt{r_+}$ and denote
\begin{equation}\label{r-TRANS}
q(z):=r_+-r(z)>0,
\qquad
q(z)\to0,
\quad\text{as }z\to+\infty.
\end{equation}

For a non-critical speed
$
c>2\sqrt{r_+},
$
the limiting linearized equation near $u=0$ and $z=+\infty$ has characteristic equation
$
\mu^2+c\mu+r_+=0
$
with two distinct negative roots
$
\mu_-<\mu_+<0.
$
The corresponding exponential decays are separated by a spectral gap.
%This exponential separation makes the asymptotic behavior relatively stable under a sufficiently small spatial perturbation $q(z)$: the leading-order decay is associated with one of the two limiting exponential decay, while the heterogeneity can be determined by using standard perturbation theory.
At the critical speed
$
c=2\sqrt{r_+},
$
the behavior is completely different.
In this case, the two characteristic roots reduce to a single root
$
-\lambda,
$ with
$
\lambda:=\sqrt{r_+}=\frac c2.
$
Hence the spectral gap disappears. The limiting homogeneous problem has the two critical decay
$
{\rm e}^{-\lambda z}
~{\rm and}~
z{\rm e}^{-\lambda z}.
$
However, for spatially heterogeneous problems, if $q(z)$ is small and tends to zero, which can effect the leading-order term.
The critical problem becomes more complex in the tail of the heterogeneous coefficient due to the loss of exponential separation.
Then we set
\begin{equation}\label{defin1}
y(z):={\rm e}^{\lambda z}\phi(z)~~{\rm and}~~\lambda=\frac{c}{2}.
\end{equation}
Then \eqref{1.3} is rewritten as
\begin{equation}\label{trans-EQ}
y''-q(z)y={\rm e}^{-\lambda z}y^2,
\end{equation}
which implies that the coefficient $q(z)$ determines the linear asymptotic structure through
\begin{equation}\label{linear-EQ}
v''-q(z)v=0,
\end{equation}
while the right-hand side
$
{\rm e}^{-cz/2}y^2
$
contains an additional exponentially decaying term, which is a nonlinear perturbation. Thus, the first step in the critical analysis is to find the fundamental solutions of \eqref{linear-EQ}.
The natural difference between integrable and non-integrable heterogeneity already appears at the linear terms.

For the equation $v''=0$,
the fundamental solutions are $1$ and $z$, which implies that the perturbations of this equation is that
$\int_{z_0}^{+\infty} zq(z)\,{\rm d}z$ for $z_0\gg1$.
Therefore, we will further consider its convergence or divergence.
We further assume that the tail of $q(z)$ is regularly varying:
\begin{equation}\label{qeqlv}
q(z)=z^{-p}L(z),
\end{equation}
where $q(z)$ satisfies \eqref{r-TRANS} and $L(z)>0$ is slowly varying at $+\infty$, namely
\begin{equation}\label{SVFLP}
\lim_{z\to+\infty}\frac{L(tz)}{L(z)}=1,\quad\text{for every }t>0.
\end{equation}
We recall the standard facts on slowly and regularly varying functions from Bingham, Goldie and Teugels \cite{BGT1987}.
In particular, $L(z)$ admits the Karamata representation (see \cite[Theorem~1.3.1]{BGT1987})
\begin{equation}\label{SVFL}
L(z)=
a(z)
\exp\left(
\int_{z_0}^{z}\frac{\varepsilon(t)}{t},{\rm d}t
\right),
\quad z\ge z_0,
\end{equation}
where
$
a(z)\to a_\infty\in(0,+\infty)$ and
$\varepsilon(z)\to0$ as $z\to+\infty$.
On the other hands, every function of thus form is slowly varying. Potter's bound \cite[Theorem~1.5.6]{BGT1987} further implies that, for every $\delta>0$,
$$
L(z)=o(z^\delta),
\quad\text{as }z\to+\infty.
$$
Further, regular variation can describe more general types of behavior. It includes, for instance,
$$
q(z)=z^{-p},
\quad
q(z)=z^{-p}(\ln z)^\alpha,
\quad
q(z)=z^{-p}(\ln z)^\alpha(\ln\ln z)^\beta,
$$
as well as many other slowly varying modifications of algebraic decay.
The importance of the exponent $p$ comes directly from the first-order criterion. Thus, we shall consider the asymptotic mechanism in three cases
$
p>2,
~
p=2
$ and
$0<p<2.$

\medskip
\noindent
 \textbf{ For the integrable case ($p>2$)}.
%$$
%q(z)=z^{-p}L(z),
%\qquad
%\int_{z_0}^{+\infty} z q(z)\,{\rm d}z<+\infty,
%$$
The homogeneous equation behaves asymptotically like
$
v''=0.
$
We obtains fundamental solutions of \eqref{linear-EQ} tend to
$
1
~\text{and}~
z.
$
Hence the tail structure is essentially retained.

\medskip
\noindent
\textbf{For the non-integrable case} ($0<p<2$).
The asymptotic behavior of the associated linear equation \eqref{linear-EQ}
is determined by
\begin{equation}\label{phase}
S(z):=\int_{z_0}^z \sqrt{q(s)}\,{\rm d}s.
\end{equation}
Lemma \ref{lem:LG-RV} in Section 2 shows that, under the weaker regularity assumption
$q(\cdot)\in C^1([z_0,+\infty))$, equation \eqref{linear-EQ} still admits a fundamental system satisfying
$
v_\pm(z)\sim q(z)^{-1/4} {\rm e}^{\pm S(z)}
$ as $z\to+\infty.$
Thus, the classical Liouville-Green asymptotic form remains valid in our paper.
Unlike the integrable case, the accumulation of $q(z)$ affects the leading order by $S(z)$. The result is a correction that decays slower than the main exponential term.

\medskip

\noindent
\textbf{For the critical case $p=2$.}
This is the transition between those two cases.
Hence the power $z^{-2}$ no longer determines whether the perturbation is the integrable or the non-integrable,
which depends on the slowly varying factor $L(z)$. Moreover, even within the non-integrable subcase, different asymptotic behaviors of $L(z)$ may generate different fundamental solutions of \eqref{linear-EQ}. This makes $p=2$ the most complex case and requires a more detailed asymptotic analysis.

In what follows, we define the linear operator
$
\mathcal L_z := \frac{{\rm d}^2}{{\rm d}z^2} - q(z)
$
with the two fundamental solutions $(v_+(z),v_-(z))$.
Let
$$
W := v_+(z)v_-'(z) - v_+'(z)v_-(z)
$$
denote the Wronskian of $v_+(z)$ and $v_-(z)$. Obviously,
$
W' = 0,
$
and hence $W$ is a nonzero constant.
Using the method of variation of constants, solving \eqref{trans-EQ} we obtain
\begin{equation}\label{Green-F}
y(z) = Av_+(z) + Bv_-(z) - \frac{v_+(z)}{W} \int_{z_0}^z v_-(s) {\rm e}^{-\lambda s}y^2(s) \, {\rm d}s + \frac{v_-(z)}{W} \int_{z_0}^z v_+(s){\rm e}^{-\lambda s}y^2(s) \, {\rm d}s.
\end{equation}
This representation will be the main tool in the nonlinear asymptotic analysis in the next section.
Regular variation provides sufficiently precise estimates for the fundamental solutions and the Green kernel. Together with a decay estimate for $\phi$, these estimates allow us to control the nonlinear term and recover the precise asymptotic behavior of the forced wave.

Most importantly, we further investigate the  asymptotic stability of forced waves for general initial data.
The proof begins with the exact equation for $v(t,z)=u(t,z)-\phi(z)$. The lower bound
$u\ge m\phi$ allows both $v$ and $-v$ to be controlled by the same
nonnegative linear solution. An exponential change of variables removes
the drift (i.e., $2\lambda u'(z)$) and introduces
$$
 V_m=q(z)+(1+m)\phi,
$$
with $q(z)$ defined by \eqref{r-TRANS},
which is nonnegative and larger than $\lambda^2$ on $z\ll-1$.
Thus, by $V_m>\lambda^2$,
the equation on the left decays strictly faster.
We use standard weak fundamental-solution theory only for the existence
and basic identities of the heat kernel (see
Aronson~\cite[Section~8, Theorem~10]{Aronson1968}). All decay estimates
needed here are obtained by comparison principle and simple calculations.
This weak framework is easy to use, since we choose the comparison function to be a step function.
The main novelty is that the analysis does not depend on a small perturbation assumption. Instead, it yields uniform convergence with explicit algebraic decay rates for general weighted initial perturbations, and further shows that under the additional decay of the initial data, one obtain faster convergence.

Our main results in this paper can be stated as follows:

\begin{theorem}\label{thmfirst1}
Let $c=2\sqrt{r_+}$ and assume that $r(z)$ satisfies (H1).
Suppose that, for some sufficiently large $z_0>0$,
$$
r_+-r(z)=z^{-p}L(z), ~~z\ge z_0,
$$
where $L(z)>0$ is slowly varying function and satisfies \eqref{SVFLP} and \eqref{SVFL}.
The bounded positive solution $\phi$ of \eqref{1.3} satisfying \eqref{in0} has the following asymptotic behavior:

\medskip
\noindent
{\bf (i) integrable case $p>2$.} If $\int_{z_0}^{+\infty} z^{1-p} L(z)\,{\rm d}z<+\infty$, then there exist constants $a>0$ and $b\in \mathbb{R}$ such that
\begin{equation*}
\phi(z)\sim \big( az+b\big){\rm e}^{-\frac{c}{2}z},~as~z\rightarrow+\infty,
\end{equation*}
where
\begin{equation}\label{conffc1}
a=\int_{-\infty}^{+\infty}{\rm e}^{\lambda s}\phi(s)\Big(q(s)+\phi(s)\Big){\rm d}s>0,\,\,\,\,
b=- \int_{-\infty}^{+\infty} s{\rm e}^{\lambda s}\phi(s)\Big(q(s)+\phi(s)\Big){\rm d}s\in\mathbb{R};
\end{equation}

\medskip
\noindent
{\bf (ii) non-integrable case $0<p<2$.} If $\int_{z_0}^{+\infty} z^{1-p} L(z)\,{\rm d}z=+\infty$, $L(\cdot)\in C^1([z_0,+\infty))$ and $z L'(z)\big{/}L(z)\to 0$ as $z\to+\infty$, then there exists a positive constant $C_+$ such that
\begin{equation*}
\phi(z) \sim C_+\Big(\frac{z^p}{L(z)}\Big)^{\frac{1}{4}}\exp\left(-\frac{c}{2} z +\int_{z_0}^zs^{-\frac{p}{2}}\sqrt{L(s)}{\rm d}s\right),~as~z\rightarrow+\infty.
\end{equation*}
 \end{theorem}

\begin{theorem}[{\bf Critical asymptotics for $p=2$}] \label{thm:critical-p2}
Let $c=2\sqrt{r_+}$ and assume that $r(z)$ satisfies (H1).
Suppose that, for some sufficiently large $z_0>0$,
$$
    r_+-r(z)= z^{-2} L(z),~~ z\ge z_0,
$$
where $L(\cdot)\in C^1([z_0,+\infty))$ is positive, slowly varying, and
monotone, and satisfies \eqref{SVFLP}-\eqref{SVFL}, and
$z L'(z)\big{/}L(z)\to 0$ as $z\to+\infty$.
The bounded positive solution $\phi$ of \eqref{1.3} satisfying \eqref{in0} has the following asymptotic behavior:
\begin{equation}\label{unifined-case}
    \phi(z)
    \sim C\,
    \sqrt{z}
    \left(\frac14+L(z)\right)^{-\frac{1}{4}}
    \exp\left\{-\frac{c}{2} z+\int_{z_0}^{z} \frac{\sqrt{1+4L(s)}}{2s}\, {\rm d}s
    \right\},~as~z\to+\infty,
\end{equation}
with constant $C>0$. Moreover, the following more explicit tails hold for constant $C_i>0(i=1,2,3,4,5)$.

\medskip
\noindent
\textnormal{ {\bf(i)} $L(z)\to0$ as $z\to+\infty$: {\bf critical integrable case.}}
If $\int_{z_0}^{+\infty} s^{-1}L(s)\,{\rm d}s<+\infty$, then
\begin{equation*}
    \phi(z)\sim C_1 z {\rm e}^{-\frac{c}{2} z},~~as~z\to+\infty.
\end{equation*}

\medskip
\noindent
\textnormal{{\bf (ii)} $L(z)\to0$ as $z\to+\infty$: {\bf critical non-integrable case.}}
If
    $\int_{z_0}^{+\infty} s^{-1}L(s)\,{\rm d}s=+\infty,$
then
\begin{equation*}
    \phi(z)
    \sim
    C_2 z \exp\left\{-\frac{c}{2} z+\int_{z_0}^{z}\frac{ L(s)}{ s}\, {\rm d}s\right \},~~as~z\to+\infty.
\end{equation*}

\medskip
\noindent
\textnormal{{\bf (iii)} $L(z)\to\ell\in(0,+\infty)$ as $z\to+\infty$.} Then
\begin{equation}\label{cri-a}
    \phi(z)
    \sim
C_3 z^{\alpha_+} \exp\left\{-\frac{c}{2} z+\int_{z_0}^{z}\frac{\sqrt{1+4L(s)}-\sqrt{1+4\ell}}{2s}\, {\rm d}s\right\},~~as~z\to+\infty,
\end{equation}
with $ \alpha_+:=\frac{1+\sqrt{1+4\ell}}{2}$.
In particular, if
$
    \int_{z_0}^{+\infty} s^{-1}|L(s)-\ell|\, {\rm d}s<+\infty,
$
then
\begin{equation}\label{asy:eq}
    \phi(z)
    \sim
    C_4 z^{\alpha_+} {\rm e}^{-\frac{c}{2}z},~~as~z\to+\infty.
\end{equation}

\medskip
\noindent
\textnormal{{\bf (iv) $L(z)\to+\infty$} as $z\to+\infty$.}
Then
\begin{equation*}
    \phi(z)
    \sim
    C_5 \Big(\frac{z^2}{L(z)}\Big)^{\frac{1}{4}}
    \exp\left\{ -\frac{c}{2} z
        + \int_{z_0}^{z}\frac{\sqrt{1+4L(s)}}{2s}\, {\rm d}s\right\},~~as~z\to+\infty.
\end{equation*}
If, in addition,
$
    \int_{z_0}^{+\infty}\frac{{\rm d}s}{s\sqrt{L(s)}}<+\infty,
$
then
\begin{equation*}
    \phi(z)
    \sim
    C_6 \Big(\frac{z^p}{L(z)}\Big)^{\frac{1}{4}}
    \exp\left\{
        -\frac{c }{2} z
        +\int_{z_0}^{z} \frac{ L(s)}{ s}\, {\rm d}s \right\},~~as~z\to+\infty.
\end{equation*}
\end{theorem}

\begin{remark}
\label{rem:L-slowly-varying}
Condition $z L'(z)\big{/}L(z)\to 0$ as $z\to+\infty$ implies that $L(z)$ is
slowly varying. Indeed, for every fixed $t>0$,
$$
    \ln\frac{L(t z)}{L(z)}
    =
    \int_z^{t z}
    \frac{sL'(s)}{L(s)}\,\frac{{\rm d}s}{s}\leq \sup\limits_{z\leq s\le t z}\left|\frac{sL'(s)}{L(s)}\right |\ln t
    \to 0,~~as~~z\to+\infty,
$$
which implies that \eqref{SVFLP} holds and is
the standard differentiable criteria for slow
variation; see \cite{BGT1987}.
\end{remark}

\begin{remark}
\label{-examples}
The asymptotic formulas in Theorems~\ref{thmfirst1} and \ref{thm:critical-p2} cover a lot of
qualitatively different decay cases.
And note that in all cases, the decay rate of the waves satisfies $\phi(z)\geq C z {\rm e}^{-cz/2}$ for some $C>0$ and $z\gg1$.
For $0<p<2$ the heterogeneity gives a
sub-exponential correction to the critical tail, whereas for $p>2$ the
classical factor $z \,{\rm e}^{-\frac{c}{2} z}$ is obtained.  For $p=2$, the
behavior is determined by the decay rate and accumulated effect of the slowly varying factor $L(z)$.

\end{remark}

In the following, we will study the asymptotic stability of forced waves of \eqref{initialvpbm}; before this,
let ${\rm BUC}(\mathbb{R})$ denote the space of bounded uniformly continuous
real-valued functions on $\mathbb{R}$. We assume that for some fixed $m\in(0,1]$,
\begin{equation}\label{eq:data}
 u_0\in{\rm BUC}(\mathbb{R}),\qquad u_0(z)\ge m\phi(z),~~z\in\mathbb{R}.
\end{equation}
and
\begin{equation}\label{eq:notation}
 \begin{gathered}
 v_0(z)=u_0(z)-\phi(z),\qquad
 W_0(z)={\rm e}^{\lambda z}v_0(z),
 \end{gathered}
\end{equation}
with $\lambda=1/2$ defined by \eqref{defin1}.
Define
\begin{equation}\label{eq:weight}
 a(z)=1+(z-z_1)_+,
 \qquad
 M_\sigma=\int_{\mathbb{R}}\big(a(z)\big)^\sigma|W_0(z)|{\rm d} z,
\end{equation}
where $(s)_+=\max\{s,0\}$, and
\begin{equation}\label{eq:X-norm}
 \|w\|_X=
 \sup_{z\le z_1}|w(z)|
 +\sup_{z\ge z_1}\frac{{\rm e}^{\lambda z}|w(z)|}{a(z)}.
\end{equation}
Throughout, $\|\cdot\|_{L^\infty(\mathbb{R})}$ denotes the uniform norm on $\mathbb{R}$.

\begin{theorem}\label{thm:stab}
Assume (H1) and let $c=2\sqrt{r_+}$. Then for any initial datum $u_0$ satisfying \eqref{eq:data}, \eqref{initialvpbm} has a unique global solution $u(t,\cdot)$ such that
\begin{equation*}
 m\phi(z)\le u(t,z+ct)\le\max\{\|u_0\|_{L^\infty(\mathbb{R})},r_+\},~~t\ge0,\ z\in\mathbb{R}.
\end{equation*}
In particular, the following statements hold.
\begin{enumerate}
\item [{\rm (i)}] If $W_0\in L^\infty(\mathbb{R})$, then there exists some constant $C_1>0$ such that
\begin{equation*}
 \|u(t,\cdot+ct)-\phi(\cdot)\|_{L^\infty(\mathbb{R})}
 \le C_1\big(\|v_0\|_{L^\infty(\mathbb{R})}+
 \|W_0\|_{L^\infty(\mathbb{R})}\big)(1+t)^{-\frac{1}{2}},~~t\ge0.
\end{equation*}
\item [{\rm (ii)}] If $0\le\sigma\le1$ and $M_\sigma<+\infty$ where $M_\sigma$ is defined in \eqref{eq:weight}, then there exists some constant $C_2>0$ such that
\begin{equation*}
 \|u(t,\cdot+ct)-\phi(\cdot)\|_{L^\infty(\mathbb{R})}+\|u(t,\cdot+ct)-\phi(\cdot)\|_X
 \le C_2(\|v_0\|_{L^\infty(\mathbb{R})}+M_\sigma)t^{-1-\frac{\sigma}{2}},~~t\ge1.
\end{equation*}
\item [{\rm (iii)}] If
 $\lim\limits_{z\to+\infty}\frac{u_0(z)}{\phi(z)}=1,$
then
$
 \lim\limits_{t\to+\infty}\|u(t,\cdot+ct)-\phi(\cdot)\|_{L^\infty(\mathbb{R})}=0.
$
\end{enumerate}
\end{theorem}

\begin{remark}
Our results cover the asymptotic stability results for the critical traveling wave of the Fisher-KPP equation considered in \cite{FH2019}. Our choice of initial data does not need to be sufficiently close to the wave profile at all times. Moreover,
the result of Theorem 1.3 (iii) has already been proved in \cite{TangWuXu}. However, in this paper we provide a new short approach.
\end{remark}

\textbf{This paper is organized as follows.}
In Section 2, we first give a complete classification of the resource functions and prove the exact decay rates of the corresponding forced waves
(Theorems 1.1-1.2). In Section 3, we prove the asymptotic stability of all forced waves and establish the convergence rates (Theorem 1.3).

\section{Spatial decay of forced waves}

In the following, we analyze the fundamental solutions of the associated linear equation \eqref{linear-EQ} in the cases $0<p<2$, $p=2$ and $p>2$, and finally use the Green-function estimates to derive the spatial decay of forced waves \eqref{1.3} satisfying \eqref{in0}.
Before this,  we first give a rough uniform upper bound for $\phi$ in a neighborhood of $+\infty$.

\begin{lemma}\label{lemma2.12new}
Under (H1), let $c=2\sqrt{r_+}$, if $\phi(z)$ is a positive solution of \eqref{1.3} satisfying \eqref{in0}, then there exist some $K>0$ and $0<\eta<\frac{c}{2}$
such that
\begin{align}\label{2.82}
|\phi(z)|\leq K {\rm e}^{-\eta z}\quad \text{for}~z\gg1,
\end{align}
\end{lemma}
\begin{proof}
Let
\begin{equation}\label{trans}
v(z) := {\rm e}^{\frac{c}{2} z} \phi(z).
\end{equation}
A direct computation gives
\begin{equation}\label{ytm10}
\phi'' (z) + c \phi' (z) + r_+ \phi (z)= {\rm e}^{-\frac{c}{2} z} v'' (z),
\end{equation}
it follows from \eqref{1.3} that
\begin{equation}\label{la0}
v''(z) = \big(q(z) + \phi(z) \big)  v(z),
\end{equation}
with $q(z) := r_+ - r(z)\geq 0$.
Note that $q(z)\to 0^+$ and $\phi(z)\to 0^+$ as $z\to+\infty$.
Choose
$$
0<\varepsilon<\frac{c^2}{4}.
$$
Then there exists $z_0\gg1$ such that
$q(z)+\phi(z)<\varepsilon $ for $z>z_0$.
Hence
\begin{equation}\label{EQ-ref}
v''(z)\leq \varepsilon v(z),\quad z>z_0.
\end{equation}
Set
\begin{equation}\label{y-ineq}
\psi(z)= v'(z)+ k v(z),
\end{equation}
where $k>0$ will be specified later.
Differentiating both sides of \eqref{y-ineq} with respect to $z$, and then applying \eqref{EQ-ref} leads to
\begin{equation*}
\begin{aligned}
\psi'(z)=v'' (z) +k v' (z) &=
v'' (z) + k \Big(\psi(z)- k v (z)\Big)\\
&\leq \Big(\epsilon - k^2\Big) v(z) + k \psi (z),~z>z_0.
\end{aligned}
\end{equation*}
We can choose $k=\sqrt{\epsilon}$ and $C_1>0$, then $\psi(z)\leq C_1 {\rm e}^{\sqrt{\epsilon} z}$ for $z\gg1$. Thus, by \eqref{y-ineq} yields that
$$
v(z)\leq \frac{C_1}{2\sqrt{\epsilon}}{\rm e}^{\sqrt{\epsilon} z}+ {\rm e}^{\sqrt{\epsilon} (z_0-z)}v(z_0),~~z>z_0,
$$
 Thus, together with \eqref{trans}, we obtain
$$
\phi(z)
={\rm e}^{-\frac{c}{2}z}v(z)
\leq
K {\rm e}^{\left(-\frac{c}{2}+\sqrt{\varepsilon}\right)z},
\,\,\,\, z>z_0.
$$
Set
$
\eta:=\frac{c}{2}-\sqrt{\varepsilon}.
$
Since $0<\varepsilon<c^2/4$, we have
$
0<\eta<\frac{c}{2},
$
which proves \eqref{2.82}.
\end{proof}

\subsection{The integrable case $p>2$}

\begin{lemma}\label{lemma1}
Under (H1) with $c=2\sqrt{r_+}$, suppose
$$
r_+-r(z) = z^{-p}L(z), \quad p > 2,
$$
where $L(z)$ is positive and slowly varying defined by \eqref{SVFLP} and \eqref{SVFL}. Then \eqref{linear-EQ}
admits fundamental solutions $v_-$ and $v_+$ such that
$$
v_-(z) = 1 + O\Big(z^{2-p}L(z)\Big), \quad v_+(z) = z + o(z),~as~z\to+\infty.
$$
which can be normalized such that
$
W(v_-,v_+):=v_+(z)v_-'(z)-v_+'(z)v_-(z)\equiv -1.
$
\end{lemma}
\begin{proof}
{\rm The proof of Lemma \ref{lemma1} is left in the Appendix A.}
\end{proof}

\begin{theorem}\label{thm:p>2-nonlinear}
Under (H1) with $c=2\sqrt{r_+}$, let
$$
r_+-r(z)=z^{-p}L(z),~\,~p>2,
$$
where $L(z)>0$ is a positive, slowly varying function defined by \eqref{SVFLP} and \eqref{SVFL}.
Then
there exist constants $a>0$ and $b\in \mathbb{R}$ such that the positive solution $y(z)$ of
\eqref{trans-EQ} satisfying
\begin{equation*}
y(z)\sim az+b,~as~z\rightarrow+\infty,
\end{equation*}
with $a$ and $b$ defined by \eqref{conffc1}.
\end{theorem}
\begin{proof}
Let $y(z)$ be the positive solution of \eqref{trans-EQ}.
For all
$
 \mu\in\left(\frac{\lambda}{2},\lambda\right),
$
we set $\eta:=2\mu-\lambda>0$, which with $y(z)={\rm e}^{\lambda z}\phi(z)$ and
Lemma \ref{lemma2.12new} gives that
\begin{equation}\label{eq:f-exp}
 f(z):={\rm e}^{-\lambda z}y^2(z)
 ={\rm e}^{\lambda z}\phi^2(z)
 \le K {\rm e}^{-(2\mu-\lambda)z}
 =K {\rm e}^{-\eta z},~~z\gg1,
\end{equation}
for some constant $K>0$ and $\lambda=c/2$.
Hence, by Lemma \ref{lemma1}
\begin{equation}\label{integrable-1}
 \int_{z_0}^{+\infty}|v_-(s)|f(s)\,{\rm d}s<+\infty,
 \qquad
 \int_{z_0}^{+\infty}|v_+(s)|f(s)\, {\rm d}s<+\infty,
\end{equation}
with $v_+(z)$ and $v_-(z)$ defined in Lemma \ref{lemma1}.
We further define
\begin{equation*}
 a:=
 A-\frac{1}{W}\int_{z_0}^{+\infty}v_-(s)f(s)\, {\rm d}s,\,\,\,\,
 b:=
 B+\frac{1}{W}\int_{z_0}^{+\infty}v_+(s)f(s)\, {\rm d}s,
\end{equation*}
for some constant $z_0>0$, which with \eqref{Green-F} gives that
\begin{equation}\label{tail-1}
 y(z)
 =
 a v_+(z)+b v_-(z)
 +\frac{v_+(z)}{W}
   \int_{z}^{+\infty}v_-(s)f(s)\, {\rm d}s
 -\frac{v_-(z)}{W}
   \int_{z}^{+\infty}v_+(s)f(s)\, {\rm d}s.
\end{equation}
In virtue of \eqref{eq:f-exp}-\eqref{integrable-1}, $v_-(z)=O(1)$, and
$v_+(z)=O(1+z)$, it follows that
\begin{align*}
 \left|
 \frac{v_+(z)}{W}\int_z^{+\infty} v_-(s)f(s)\, {\rm d}s
 \right|
 &=
 O((1+z){\rm e}^{-\eta z}),\\
 \left|
 \frac{v_-(z)}{W}\int_z^{+\infty} v_+(s)f(s)\,{\rm d}s
 \right|
 &=
 O((1+z){\rm e}^{-\eta z}).
\end{align*}
Therefore \eqref{tail-1} leads to
\begin{equation}\label{AsY-EQ}
 y(z)=a v_+(z)+b v_-(z)+o(1),~~{\rm as}~z\to+\infty.
\end{equation}

Next, we shall compute the values of a and b.
Since
$r_+-r(z)=:q(z)=z^{-p}L(z),\,~p>2,
$
by Karamata's theorem \cite[Theorem 1.5.11]{BGT1987} yields
$
\int_{z_0}^{+\infty} z q(z)\,{\rm d}z<+\infty$,
which with \eqref{AsY-EQ} and $\phi(-\infty)=r_-$ imply that
\begin{equation*}
\int_{-\infty}^{+\infty}y(z)\Big\{q(z)+\phi\Big\}{\rm d}z<+\infty.
\end{equation*}
Integrating \eqref{trans-EQ} from $-\infty$ to $z$ twice and
interchanging the order of integration, we get
\begin{equation*}
\begin{aligned}
y(z)
=&\int_{-\infty}^z\int_{-\infty}^xy(t)\Big(q(t)+\phi(t)\Big){\rm d}t\\
=&\int_{-\infty}^z\big(z-t\big)y(t)\Big(q(t)+\phi(t)\Big){\rm d}t\\
=& z\int_{-\infty}^{+\infty}y(t)\Big(q(t)+\phi(t)\Big){\rm d}t
- \int_{-\infty}^{+\infty} ty(t)\Big(q(t)+\phi(t)\Big){\rm d}t \\
&- \int_z^{+\infty} \big(z-t\big) y(t)\Big(q(t)+\phi(t)\Big){\rm d}t=a+b+o(1),
\end{aligned}
\end{equation*}
with
$
\int_z^{+\infty} \big(z-t\big) {\rm e}^{\lambda t}\phi(t)\big(q(t)+\phi(t)\big){\rm d}t=o(1)~{\rm as}\,\,z\to+\infty,
$
this completes the proof.
\end{proof}

\subsection{The non-integrable case $0<p<2$}

\begin{lemma}
\label{lem:LG-RV}
Under (H1) with $c=2\sqrt{r_+}$, let
\begin{equation}\label{q(z)-Eq}
r_+-r(z)=z^{-p}L(z), \,\,\,\,0<p<2,
\end{equation}
and $L(z)$ is a positive, slowly varying function defined by \eqref{SVFLP} and \eqref{SVFL}. Assume further that
\begin{equation}\label{L(z)-Eq}
L(\cdot)\in C^1([z_0,+\infty))
\,\,\,\,and\,\,\,\,
\frac{zL'(z)}{L(z)}\to 0,
\,\,\text{as }z\to+\infty.
\end{equation}
Then \eqref{linear-EQ}
admits fundamental solutions $v_+(z)$ and $v_-(z)$
such that
\begin{equation}\label{vpm-Eq}
v_\pm(z)
=
\Big(\frac{z^p}{L(z)}\Big)^{\frac{1}{4}}
\exp\left(
\pm\int_{z_0}^{z}s^{-\frac{p}{2}}\sqrt{L(s)}\, {\rm d}s
\right)
(1+o(1)),
\quad\text{as }z\to+\infty,
\end{equation}
which can be normalized such that
$
W:=v_+(z)v_-'(z)-v_+'(z)v_-(z)\equiv-2.
$
\end{lemma}
\begin{proof}
{\rm The proof of Lemma \ref{lem:LG-RV} is left in the Appendix B.}
\end{proof}

\begin{remark}
A standard formulation of the classical Liouville-Green approximation theorem assumes that
$q(z)\in C^2([z_0,+\infty))$ together with the integrability condition
$$
\int_{z_0}^{+\infty}\left|\frac{5(q')^2}{16 q^{5/2}}-\frac{q''}{4 q^{3/2}}\right|{\rm d}z<+\infty,
$$
which can be found in \cite[Definition 12.2]{HK2021}.
In Lemma \ref{lem:LG-RV}, we only assume $q(\cdot)\in C^1([z_0,+\infty))$ and $zL'(z)/L(z)\to 0$ as $z\to+\infty$, the same asymptotic fundamental solutions are obtained under this weaker regularity assumption. This is possible because the leading-order behavior of the solutions is determined by the decay rate of $q$, and the additional differentiability required in the standard approach is not essential.
\end{remark}

In the following,
we will define a function $Y(z)$ and apply Green's function representation to establish a priori estimates and
asymptotic behavior for the solutions.
Define
\begin{equation}\label{Weight-EQ}
Y(z):=\frac{y(z)}{v_+(z)},
\end{equation}
which with $W=-2$ that \eqref{Green-F} becomes
\begin{equation}\label{Green-p}
\begin{aligned}
Y(z)
&=
A+B\frac{v_-(z)}{v_+(z)}
+\frac12
\int_{z_0}^z
{\rm e}^{-\lambda s}Y^2(s)v_-(s)v^2_+(s)\,{\rm d}s
-\frac12\frac{v_-(z)}{v_+(z)}
\int_{z_0}^z
{\rm e}^{-\lambda s}Y^2(s)v^3_+(s)\, {\rm d}s.
\end{aligned}
\end{equation}
with $v_+(z)$ and $v_-(z)$ defined by \eqref{vpm-Eq} and $\lambda=c/2$.

\begin{lemma}\label{lemma-Yz}
Under (H1) with $c=2\sqrt{r_+}$, there exist $z_0\gg1$ such that
\begin{equation*}
K(s):=
{\rm e}^{-\lambda s}\big(q(s)\big)^{-\frac34}{\rm e}^{S(s)}
\in L^1([z_0,+\infty)),
\end{equation*}
where $q(s)$ and $S(s)$ are defined by \eqref{r-TRANS} and  \eqref{phase} respectively.
Moreover,
\begin{equation}\label{ineq-est1}
v_-(s)v^2_+(s)
\le C\big(q(s)\big)^{-\frac34} {\rm e}^{S(s)}\,\,\,\,and\,\,\,\,
\frac{v_-(z)}{v_+(z)}v^3_+(s)
\le C\big(q(s)\big)^{-\frac34} {\rm e}^{S(s)},~~z\geq z_0,
\end{equation}
with $v_+(\cdot)$ and $v_-(\cdot)$ defined by \eqref{vpm-Eq}.
\end{lemma}

\begin{proof}
By virtue of \eqref{vpm-Eq}, there exists some constant $C>0$ such that
$$
v_-(s)v^2_+(s)
\le C\big(q(s)\big)^{-\frac34} {\rm e}^{S(s)},
~~
\frac{v_-(z)}{v_+(z)}
\le C {\rm e}^{-2S(z)}
~~ {\rm and}~~
v^3_+(s)\le C\big(q(s)\big)^{-\frac34} {\rm e}^{3S(s)}.
$$
Since $S(\cdot)$ is increasing and $s\le z$, then
$
-2S(z)+3S(s)\le S(s),
$
which proves \eqref{ineq-est1}.
By Potter's bound for slowly varying functions
\cite[Theorem~1.5.6]{BGT1987},
for every $\delta>0$ there exists $C_\delta>0$ such that $L(s)^{-3/4}\leq C_\delta s^{\delta}$, which by \eqref{q(z)-Eq} yields that
$$
\big(q(s)\big)^{-\frac34}=s^{\frac 3p4}L(s)^{-\frac 34}\le C_\delta s^{\frac {3p}{4}+\delta},~~s\geq z_0.
$$
Since  $S(s)=o(s)$, we can find
$z_1$ large enough such that $S(s)\le \frac{\lambda}{2}s$ for $s\ge z_0$.
Hence
$$
K(s)
\le
\widetilde{C}_\delta {\rm e}^{-\frac {\lambda s}{2}} s^{\frac{3p}{4}+\varepsilon},~~s\ge z_0,
$$
for some $\widetilde{C}_\delta>0$, which is integrable on $[z_0,+\infty)$.
\end{proof}

\begin{theorem}\label{lem:Y-bounded}
Under (H1) with $c=2\sqrt{r_+}$, let
$$
q(z)=z^{-p}L(z),~\,~0<p<2,
$$
where $L(z)>0$ is a positive, slowly varying function defined by \eqref{SVFLP} and \eqref{SVFL}.
Further assume  that
$
L(\cdot)\in C^1([z_0,+\infty))
\,\,\,\,and\,\,\,\,
\frac{zL'(z)}{L(z)}\to 0
\,\,\text{as }z\to+\infty.
$
Then there exists a finite constant $C_+>0$ such that
\begin{equation}
\frac{y(z)}{v_+(z)}\to C_+,~~as~z\to+\infty.
\end{equation}
\end{theorem}

\begin{proof}
It follows from Lemma \ref{lemma-Yz} that $K(\cdot)\in L^1([z_0,+\infty))$. Thus, we have
\begin{equation}\label{abg}
C_1\int_{z_0}^\infty K(s)\, {\rm d}s
\le
\frac{1}{2C_0},
\end{equation}
for some constant $C_0>0$.
Define
$$
U(z)
:=
C_0+
C_1\int_{z_0}^z K(s)|Y(s)|^2\, {\rm d}s,
\quad z\ge z_0.
$$
Further, Lemma \ref{lemma-Yz} implies that $|Y(z)|\le U(z)$ for $z\ge z_0$.
Clearly, $U(z)\ge C_0>0$.
Moreover, $U$ is continuous and, for every $z\ge z_0$,
$$
U'(z)
=
C_1K(z)|Y(z)|^2
\le
C_1K(z)U^2(z).
$$
Therefore
\begin{equation}\label{ine_EQ0}
\left(\frac{1}{U(z)}\right)'
=
-\frac{U'(z)}{U^2(z)}
\ge
-C_1K(z), ~~z\ge z_0.
\end{equation}
Integrating \eqref{ine_EQ0} from $z_0$ to
$z$ gives
$$
\frac{1}{U(z)}
\ge
\frac{1}{U(z_0)}
-
C_1\int_{z_0}^z K(s)\, {\rm d}s.
$$
Since $U(z_0)=C_0$, it follows from \eqref{abg} that
$
U(z)\le 2C_0
$  for $z\ge z_0$, which means $|Y(z)| \le 2C_0$ for $z\ge z_0$.
By \eqref{ineq-est1} and boundedness of $Y$, we have
$$
\int_{z_0}^{+\infty}
{\rm e}^{-\lambda s}Y^2(s)v_-(s)v^2_+(s)\, {\rm d}s<+\infty.
$$
Further define
$$
C_+
=
A-\frac1 2
\int_{z_0}^{+\infty}
{\rm e}^{-\lambda s}Y^2(s)v_-(s)v^2_+(s)\, {\rm d}s,
$$
with $A>0$ defined by \eqref{Green-F},
which with \eqref{Green-p} implies that $Y(z)\to C_+$ as $z\to+\infty$ for $C_+\ge0$.

It remains to show that $C_+>0$.
By contradiction, suppose that $C_+ = 0$. It follows from \eqref{vpm-Eq} that $v_-(z) \to 0$ , which means
$y(z) \to 0 $ as $z \to +\infty$.
On the other hand,
$$
y(z) = {\rm e}^{\lambda z} \phi(z) \to 0, \quad \text{as } z \to -\infty,
$$
since $\phi(-\infty) = r_- > 0$. Moreover $y > 0$ on $\mathbb{R}$. Hence $y$ attains a positive global maximum at some $z_m \in \mathbb{R}$. Consequently,
$$
y'(z_m) = 0, \quad y''(z_m) \leq 0.
$$
Since $q \geq 0$, $\phi > 0$, and $y > 0$,
$$
y''(z_m) = (q(z_m) + \phi(z_m))y(z_m) > 0,
$$
which contradicts to \eqref{trans-EQ}. Thus $C_+ \neq 0$. We have
$
C_+ > 0.
$
\end{proof}

{\bf Proof of Theorem \ref{thmfirst1}.}
Since
$
y(z)={\rm e}^{\frac{c}{2}z}\phi(z),
$
Case (i) of Theorem~\ref{thmfirst1} follows immediately from Theorem~\ref{thm:p>2-nonlinear}. Furthermore, combining \eqref{vpm-Eq} with Theorem~\ref{lem:Y-bounded}, we prove case~(ii). This completes the proof.

\subsection{The critical case $p=2$}
In this sub-section, using \eqref{qeqlv},  we rewrite \eqref{linear-EQ} as
\begin{equation} \label{eq:p2-linear}
    v''(z)-\frac{L(z)}{z^2}v(z)=0,
    \qquad z\ge z_1.
\end{equation}
The critical structure is closely related to the perturbed
Euler equations. Regularly varying solutions of such equations,
including generalized regular variation on the logarithmic scale,
have been extensively studied; see,
\cite{KM2010,Maric2000}.
Our purpose here is different: we need explicit asymptotic formulas for fundamental solutions, since these decay rates will be
used later in the nonlinear variation-of-constants representation.

\begin{lemma}\label{lem:3.1}
Under (H1) with $c=2\sqrt{r_+}$, let $L(\cdot)\in C^1([z_0,+\infty))$ is a positive, slowly varying function defined by \eqref{SVFLP} and \eqref{SVFL}.
Assume that $L(z)$ is monotone for sufficiently large $z$, and that
$
\frac{zL'(z)}{L(z)}\to0
$ as $z\to+\infty$.
Then \eqref{eq:p2-linear} admits fundamental solutions $v_+$ and $v_-$ satisfying
\begin{equation}\label{eq:p2-unified}
\begin{split}
    v_\pm(z)
    &=
    \sqrt{z}
    \left(\frac14+L(z)\right)^{-\frac 14}
    \exp\left\{
        \pm
        \int_{z_0}^{z} s^{-1}\sqrt{\frac14+L(s)}\,{\rm d}s
    \right\}(1+o(1)),~~as~~z\to+\infty,
\end{split}
\end{equation}
which can be normalized such that
$
W:=v_+(z)v_-'(z)-v_+'(z)v_-(z)\equiv-2.
$
\end{lemma}
\begin{proof}
{\rm The proof of Lemma \ref{lem:3.1} is left in the Appendix C.}
\end{proof}

\begin{remark}
In the proof of Lemma \ref{lem:3.1}, $C^1$-regularity is sufficient for the asymptotics.
The proof is based on the classical asymptotic-integration argument of Hartman--Wintner type;
see \cite{HW1955}. Furthermore, The monotonicity assumption is used only in proving the
integrability condition \eqref{estimate-EQ}. Without
monotonicity, the condition
$
zL'(z)/L(z)\to0
$ as $z\to+\infty$
does not control the higher perturbation.
\end{remark}

%\subsubsection{Refined asymptotic classification}

We next consider more explicit asymptotics of
Lemma~\ref{lem:3.1}.
Since $L(z)$ is monotone for $z\gg1$, the limit
$$
    L_{\infty}:=\lim_{z\to+\infty}L(z)
$$
exists in $[0,+\infty]$. In the
following, $C_\pm>0$ denotes a constant which may change from one case to
another.

\begin{corollary}
\label{cor:p2-classification}
Under the assumptions of
Lemma~\ref{lem:3.1}, the following special cases hold.

\medskip
\noindent
\textnormal{\rm {\bf (i) $L(z)\to0$}} as $z\to+\infty$.
Then,
\begin{equation*}
    v_+(z)\sim C_+
    z {\rm e}^{B(z)},\qquad
    v_-(z)\sim C_-
    {\rm e}^{-B(z)},~~as~z\to+\infty,
\end{equation*}
with $ B(z) := \int_{z_0}^{z}\frac{\sqrt{1+4L(s)}-1 }{2s}\, {\rm d}s.$
More precisely:

\begin{enumerate}
\item[\textnormal{\rm {\bf (i-a)}}]
If
    $\int_{z_0}^{+\infty} s^{-1}L(s)\,{\rm d}s<+\infty$,
then the fundamental solutions satisfy
\begin{equation*}
    v_+(z)\sim C_+z,
    \qquad
    v_-(z)\sim C_-,~~{\rm as}~z\to+\infty.
\end{equation*}

\item[\textnormal{\rm {\bf (i-b)}}]
If
   $
    \int_{z_0}^{+\infty} s^{-1}L(s)\, {\rm d}s=+\infty,$
then $B(z)\sim\int_{z_0}^{z}s^{-1}L(s)\, {\rm d}s$ and
\begin{equation*}
    v_+(z)
    \sim C_+z {\rm e}^{B(z)},
    \qquad
    v_-(z)
    \sim C_-{\rm e}^{-B(z)},~~as~z\to+\infty.
\end{equation*}
\end{enumerate}
In addition, if
$
    \int_{z_0}^{+\infty}s^{-1}L^2(s)\, {\rm d}s<+\infty,
$
then
\begin{equation*}
    v_+(z)
    \sim
    C_+ z\exp\left\{
        \int_{z_0}^{z} s^{-1}L(s)\, {\rm d}s \right\},
    ~~
    v_-(z) \sim
    C_-\exp\left\{ -\int_{z_0}^{z} s^{-1}L(s)\, {\rm d}s\right\},~as~z\to+\infty.
\end{equation*}

\medskip
\noindent
\textnormal{\rm {\bf (ii) $L(z)\to\ell\in(0,+\infty)$}} as $z\to+\infty$.
Then
\begin{equation}\label{caseii}
    v_\pm(z)
    \sim
    C_\pm
    z^{\alpha_\pm}
    \exp\left\{\int_{z_0}^{z} s^{-1}\left( \sqrt{\frac14+L(s)}
        - \sqrt{\frac14+\ell}\right)\, {\rm d}s\right\},~~as~z\to+\infty,
\end{equation}
with
$
\alpha_\pm:=\frac{1\pm\sqrt{1+4\ell}}2.
$
In particular, if
$
    \int_{z_0}^{+\infty} s^{-1}|L(s)-\ell|\, {\rm d}s<+\infty,
$
then
\begin{equation*}
    v_\pm(z)
    \sim
    C_\pm z^{\alpha_\pm},~~as~z\to+\infty.
\end{equation*}

\medskip
\noindent
\textnormal{\rm {\bf (iii) $L(z)\to+\infty$}} as $z\to+\infty$.
Then
\begin{equation*}
\begin{split}
    v_\pm(z)
    &\sim C_\pm \left(\frac{z^2}{L(z)}\right)^{\frac{1}{4}}
    \exp\left\{
        \pm\int_{z_0}^{z} s^{-1}\sqrt{L(s)}\, {\rm d}s\pm D(z)
    \right\},~~as~z\to+\infty,
\end{split}
\end{equation*}
with $ D(z)
    :=\int_{z_0}^{z} s^{-1}\Big(
        \sqrt{\frac14+L(s)}-\sqrt{L(s)}\Big)\, {\rm d}s$.
Further if
$
    \int_{z_0}^{+\infty} \frac{{\rm d}s}{s\sqrt{L(s)}}<+\infty,
$
then
\begin{equation*}
    v_\pm(z)
    \sim C_\pm
   \left(\frac{z^2}{L(z)}\right)^{\frac{1}{4}} \exp\left\{
        \pm\int_{z_0}^{z} s^{-1}\sqrt{L(s)}\, {\rm d}s
    \right\},~~as~z\to+\infty.
\end{equation*}
\end{corollary}

\begin{proof}
We only check the proof from the unified formula
\eqref{eq:p2-unified}.

If $L(z)\to0$, then
$
    \left(1/4+L(z)\right)^{-1/4}
    \to \sqrt2,
$
and
$$
    \int_{z_0}^{z} s^{-1}\sqrt{\frac14+L(s)}\, {\rm d}s
    =
    \frac12\ln z+E(z)+C,
$$
with $ E(z) := \int_{z_0}^{z}s^{-1}\big( \sqrt{1/4+L(s)}-1/2 \big)\, {\rm d}s.$
By normalizing the fundamental solutions, case (i) is proved.
Furthermore,
\begin{equation*}
    \sqrt{\frac14+L(s)}-\frac12
    =
    \frac{L(s)}{
        \sqrt{\frac14+L(s)}+\frac12}\sim L(s).
\end{equation*}
Hence $B(z)$ converges if and only if
$\int_{z_0}^{+\infty}s^{-1} L(s)\, {\rm d}s$ converges.  This proves cases (i)-a and (i)-b.

If $L(z)\to\ell\in(0,+\infty)$, using Taylor expansion, we obtain
$$
    \int_{z_0}^{z}
    \frac{\sqrt{1+4L(s)}}{2s}\, {\rm d}s
    =\frac{\sqrt{1+4\ell}}{2}\ln z
    +O(1),~~{\rm as}~z\to+\infty.
$$
Since
$
    \left(1/4+L(z)\right)^{-1/4}
    \to
    \left(1/4+\ell\right)^{-1/4},
$
asymptotic behavior \eqref{caseii} holds.
Moreover, note that
$$
    \left|
        \sqrt{\frac14+L(s)}
        -
        \sqrt{\frac14+\ell}
    \right|
    \le
     \frac{2|L(s)-\ell|}{\sqrt{1+4L(s)}+\sqrt{1+4\ell}}\leq C|L(s)-\ell|,~~s\gg1,
$$
then if
$
    \int_{z_0}^{+\infty} s^{-1}|L(s)-\ell|\, {\rm d}s<+\infty
$, which completes the proof of case (ii).

Finally, if $L(z)\to+\infty$, then
$
    \left(1/4+L(z)\right)^{-1/4}
    \sim
    L(z)^{-1/4},
$
and
$$
    \int_{z_0}^{z} s^{-1}\sqrt{1/4+L(s)}\,{\rm d}s
    =\int_{z_0}^{z}s^{-1}\sqrt{L(s)}\,{\rm d}s
    +D(z).
$$
Since
\begin{equation*}
    \sqrt{\frac14+L(z)}-\sqrt {L(z)}
    =
    \frac{1/4}{
        \sqrt{1/4+L(z)}+\sqrt {L(z)}}
    \sim\frac{1}{8\sqrt {L(z)}},~~{\rm as}~z\to+\infty,
\end{equation*}
the correction $D(z)$ converges if and only if
$
    \int_{z_0}^{+\infty}
    \frac{{\rm d}s}{s\sqrt{L(s)}}<+\infty.
$
This proves case (iii) and completes the proof.
\end{proof}

We now return to the nonlinear equation \eqref{trans-EQ} with $q(z)=\frac{L(z)}{z^2}$.

\begin{lemma}
\label{lem:p2-subexponential}
For every $\mu>0$, there exists $z_0\gg1$ such that
\begin{equation}\label{eq:p2-3.5}
    |v_+(z)|+|v_-(z)|
    \le C_\mu {\rm e}^{\mu z},
    \qquad z\ge z_0,
\end{equation}
with $v_\pm$ defined by \eqref{eq:p2-unified}.
Moreover,
$
    \frac{v_-(z)}{v_+(z)}
    \to0
    ~\text{as }z\to+\infty.
 $
\end{lemma}

\begin{proof}
Since $L(z)$ is slowly varying, then
$
L(z)=o(z^\delta)~\text{as }z\to+\infty
$
for every $\delta>0$. Thus,
$
\sqrt{1/4+L(z)}=o(1/2+z^{\delta/2}),
$
and hence
$$
\int_{z_0}^{z} s^{-1}\sqrt{1/4+L(s)}\,{\rm d}s
= o(z),~~{\rm as}~z\to+\infty.
$$
Together with
$
\ln\left(z^{1/2}\left(1/4+L(z)\right)^{-1/4}\right)=o(z)~{\rm as}~z\to+\infty,
$
this proves \eqref{eq:p2-3.5}. On the other hand, by \eqref{eq:p2-unified},
$$
\frac{v_-(z)}{v_+(z)}=\exp\left\{-\int_{z_0}^{z} \frac{\sqrt{1+4L(s)}}{s}\,{\rm d}s
\right\}(1+o(1)).
$$
Note that
$
\sqrt{1/4+L(s)}\ge 1/2,
$
then Lemma \ref{lem:p2-subexponential} is completed.
\end{proof}

\medskip

We next use the same Green-function argument as in the case
$0<p<2$.
Set
$
Y(z):=\frac{y(z)}{v_+(z)},
$
which also satisfies \eqref{Green-p}.
Thus, define
\begin{equation}
    K_1(s)
    :={\rm e}^{-\lambda s}
    |v_-(s)|\,|v_+(s)|^2,~~{\rm and}~~
    K_2(s)
    :={\rm e}^{-\lambda s}|v_+(s)|^3.
\label{eq:p2-kernels}
\end{equation}

\begin{lemma}
\label{lem:p2-kernel-L1}
Under the assumptions of
Lemma~\ref{lem:3.1}, one has
\begin{equation}\label{p2-kernel-L1}
    K_1,K_2\in L^1([z_0,+\infty))
\end{equation}
for some sufficiently large $z_\mu$. In additional, there exists a finite constant $C_+>0$ such that
$$
    \frac{y(z)}{v_+(z)}
    \to C_+,~~as~z\to+\infty.
$$
\end{lemma}

\begin{proof}
Choose $\mu>0$ small such that
$3\mu<\lambda$.
By Lemma~\ref{lem:p2-subexponential}, then
$
    K_1(s)+K_2(s)  \le  C {\rm e}^{-(\lambda-3\mu)s},
$
which proves \eqref{p2-kernel-L1}.
Moreover, it follows from \eqref{p2-kernel-L1} that \eqref{Green-p} has following estimate
\begin{equation*}
    |Y(z)|
    \le
    C+ C\int_{z_0}^z
    K(s)|Y(s)|^2\,{\rm d}s,
    \qquad
    K:=K_1+K_2\in L^1([z_0,+\infty)),
\end{equation*}
By repeating the nonlinear estimates established in Theorem \ref{lem:Y-bounded} for the case $0<p<2$, we complete the proof of Lemma \ref{lem:p2-kernel-L1}. Finally, recalling
$
y(z)={\rm e}^{\frac{c}{2}z}\phi(z)
$
and Corollary \ref{cor:p2-classification}, the proof of Theorem \ref{thm:critical-p2} follows immediately.
\end{proof}

\section{Asymptotic stability of forced waves}
This section is devoted to the asymptotic stability of solution $u(t,x)$ of (\ref{initialvpbm}) in
moving coordinate $z=x-ct$ with more general bounded non-negative  initial data $u_0(z)$,
here $u(t,z)$ satisfies the following initial value problem
\begin{equation}\label{eq:initi865}
\begin{cases}
 & u_t=u_{zz}+2\lambda u_z+u(r(z)-u),~~z\in \mathbb{R},~t>0,\\
 & u(0,z)=u_0(z),~~~~~~~~~~~~~~~~~~~~~~z\in\mathbb{R}, \\
\end{cases}
\end{equation}
where $u_0$ satisfies \eqref{eq:data} and $\lambda=c/2$ defined by \eqref{defin1}.

\begin{lemma}\label{lemma3.1-Eq}
Assume (H1) and let $c=2\sqrt{r_+}$. Then for any initial datum $u_0$ satisfying \eqref{eq:data}, \eqref{eq:initi865} has a unique global solution $u(t,\cdot)$ such that
\begin{equation}\label{ineq:max}
 m\phi(z)\le u(t,z)\le\max\{\|u_0\|_{L^\infty(\mathbb{R})},r_+\},~~t\ge0,\ z\in\mathbb{R}.
\end{equation}
\end{lemma}
\begin{proof}
Let
$$
 M=\max\{\|u_0\|_{L^\infty(\mathbb{R})},r_+\}.
$$
The constant $M$ is a stationary supersolution because
$M(r(z)-M)\le0$. Moreover,
\begin{align*}
 &(m\phi)''+2\lambda(m\phi)'+m\phi(r(z)-m\phi)\notag\\
 =&m\bigl[\phi''+2\lambda\phi'+\phi(r(z)-\phi)\bigr]
   +m(1-m)\phi^2
 =m(1-m)\phi^2\ge0.
\end{align*}
The last inequality holds for $0<m\leq 1$.
Thus $m\phi$ is a stationary subsolution. Applying the comparison principle to \eqref{eq:initi865}, we obtain
\eqref{ineq:max} and $u(t,z)$ is global. This completes the proof.
\end{proof}

In the following, fix $m\in(0,1]$, since $c=2\sqrt{r_+}=2\lambda$, then set
\begin{equation}\label{eq:comp1}
 v(t,z):=u(t,z)-\phi(z),
\end{equation}
and note that
$$
 u(r(z)-u)-\phi(r(z)-\phi)=(r(z)-u-\phi)(u-\phi).
$$
Combining this with \eqref{1.3} and \eqref{eq:initi865}, we see that the perturbation
$v(t,z)$ satisfies
\begin{equation}\label{eq:coma2}
 \begin{cases}
 v_t=v_{zz}+2\lambda v_z+
       \bigl(r(z)-u(t,z)-\phi(z)\bigr)v,
       &t>0,\\
 v(0,z)=v_0(z),&z\in\mathbb{R}.
 \end{cases}
\end{equation}
Next, we shall define the exponentially weighted perturbation by
\begin{equation}\label{weighted-pertur}
 W(t,z):={\rm e}^{\lambda z}v(t,z),
\end{equation}
further by $q(z)=\lambda^2-r(z)$,
then \eqref{eq:coma2} can be rewritten as
\begin{equation}\label{we2}
 W_t=W_{zz}-\Big(q(z)+u(t,z)+\phi(z)\Big)W,~~t>0,~z\in\mathbb{R},
\end{equation}
where $v_0$ and $W_0$ satisfy \eqref{eq:notation}, and $W_0$ is not assumed to be bounded.

\subsection{A linear upper bound for $v$ and $W$}

In this sub-section, we will obtain a linear upper bound for $v$ and $W$ of \eqref{eq:coma2} and \eqref{weighted-pertur}, respectively. Before this,
we first give some standard linear results needed in the comparison principle.
For a real-valued $V(\cdot)\in L^\infty(\mathbb{R})$ with $V\ge0$, let
$K_V(t,x,y)$ be the nonnegative fundamental solution of
\begin{equation}\label{linear-equation-2}
 w_t=w_{xx}-V(x)w.
\end{equation}
It follows from the theory of weak fundamental solutions; see
\cite[Section~8, Theorem~10]{Aronson1968} that for bounded measurable initial data
$g$, the solution is represented by
$$
 w(t,x)=\int_{\mathbb{R}}K_V(t,x,y)g(y){\rm d} y.
$$
Further it is easy to check that
\begin{equation}\label{eq:symmetry}
 K_V(t,x,y)=K_V(t,y,x),~~ K_V(t+s,x,y)
 =\int_{\mathbb{R}}K_V(t,x,\xi)K_V(s,\xi,y){\rm d}\xi,~ t,s>0.
\end{equation}
The sign of $V(x)$ will provide a precise estimate. Set
$$
 M:=\|V\|_{L^\infty(\mathbb{R})},
 \qquad
 G(t,x):=\frac{1}{\sqrt{4\pi t}}
             \exp\!\left\{-\frac{x^2}{4t}\right\}.
$$
For $g\in C_c^\infty(\mathbb{R})$ with $g\ge0$, let
$$
 \widetilde{u}(t,x)=\int_{\mathbb{R}}G(t,x-y)g(y){\rm d} y~~~{\rm and}~~~
 u_M(t,x)={\rm e}^{-Mt}\widetilde{u}(t,x).
$$
Therefore, we define linear operator $\mathcal L_V:=\partial_t-\partial_{xx}+V(x)$, using the fact that
$\mathcal L_V \widetilde{u}=V\widetilde{u}\ge0$, we obtain
$$
 \mathcal L_V u_M=(V-M)u_M\le0.
$$
Comparied with the solution of \eqref{linear-equation-2}
with initial datum $g$, this gives
\begin{equation}\label{eq:f-bound}
 {\rm e}^{-Mt}G(t,x-y)\le K_V(t,x,y)\le G(t,x-y),
 ~~~ t>0,\, x,y\in\mathbb{R}.
\end{equation}
By the comparison principle for \eqref{linear-equation-2}, if
$V_1,V_2\in L^\infty(\mathbb{R})$ and $V_1\ge V_2\ge0$, then $V_1$ is a subsolution for $V_2$. Therefore,
\begin{equation}\label{kernel-comparison}
 K_{V_1}(t,x,y)\le K_{V_2}(t,x,y).
\end{equation}
Define
\begin{equation}\label{eq:general}
 S_V(t,x):=\int_{\mathbb{R}}K_V(t,x,y){\rm d} y.
\end{equation}
It is the bounded weak solution with constant initial datum $1$:
\begin{equation}\label{eq:-genl2}
 (S_V)_t=(S_V)_{xx}-V(x)S_V,
 \qquad S_V(0,x)=1.
\end{equation}
Integrating \eqref{eq:f-bound} yields
\begin{equation}\label{eq:mass-basic-bound}
 {\rm e}^{-Mt}\le S_V(t,x)\le1,
 \qquad
 \|S_V(t,\cdot)-1\|_\infty\le1-{\rm e}^{-Mt}.
\end{equation}
In particular, the initial condition in
\eqref{eq:-genl2} holds uniformly in $x$.
By \eqref{eq:symmetry}, \eqref{eq:general} and $0\le S_V\le1$, we also have
\begin{equation}\label{eq:monotone}
 \begin{aligned}
 S_V(t+s,x)
 =\int_{\mathbb{R}}K_V(t,x,\xi)S_V(s,\xi){\rm d}\xi
 \le\int_{\mathbb{R}}K_V(t,x,\xi){\rm d}\xi=S_V(t,x).
 \end{aligned}
\end{equation}
Thus $t\mapsto S_V(t,x)$ is nonincreasing for each fixed $x$.
The exponential change of variables removes the drift from \eqref{eq:coma2}, which obtains \eqref{we2}.
However, its effect on the initial datum must be analyzed separately, since ${\rm e}^{\lambda z}v_0(z)$ may be unbounded.

Next, we will construct upper bounds for $v(z)$ and $W(z)$ by using the weak comparison principle.
Thus, we define
\begin{equation}\label{eq:Km-Sm}
 V_m=q(z)+(1+m)\phi,~~ K_m(t,z,y):=K_{V_m}(t,z,y),~~
 S_m(t,z):=\int_{\mathbb{R}}K_m(t,z,y){\rm d} y.
\end{equation}

\begin{lemma}\label{lem:linear-up}
Let $f\in{\rm BUC}(\mathbb{R})$ be
nonnegative and define
\begin{equation}\label{eq:h-definition}
 h_f(t,z):={\rm e}^{-\lambda z}
       \int_{\mathbb{R}}K_m(t,z,y){\rm e}^{\lambda y}f(y){\rm d} y,
       \,\,\, t>0,
\end{equation}
which satisfies
\begin{equation}\label{uni-bound}
 0\le h_f(t,z)\le {\rm e}^{\lambda^2t}\|f\|_{L^\infty(\mathbb{R})}.
\end{equation}
Moreover, $h_f$ is the unique weak solution of
\begin{equation}\label{eq:h-pde}
 \begin{cases}
 (h_f)_t=(h_f)_{zz}+2\lambda (h_f)_z
       +\bigl(r(z)-(1+m)\phi(z)\bigr)h_f,&t>0,\\
 h_f(0,z)=f(z),&z\in\mathbb{R},
 \end{cases}
\end{equation}
If \eqref{eq:data} and \eqref{eq:notation} also hold, then
\begin{align}
 |v(t,z)|&\le h_{|v_0|}(t,z),
 \label{lemma-v}\\
 |W(t,z)|&\le\int_{\mathbb{R}}K_m(t,z,y)|W_0(y)|{\rm d} y.
 \label{lemma-W}
\end{align}
\end{lemma}

\begin{proof}
Let $f\in{\rm BUC}(\mathbb{R})$ be nonnegative and satisfy \eqref{eq:h-definition}.
Applying the Gaussian approximate-identity property gives $ \|h_f(t,\cdot)-f(\cdot)\|_{L^\infty(\mathbb{R})}\to0$ uniformly as
$t\downarrow0$.
Further by \eqref{eq:f-bound},
\begin{align}\label{eq:h-finite}
 h_f(t,z)
 &\le \|f\|_{L^\infty(\mathbb{R})}
 \int_{\mathbb{R}}G(t,z-y){\rm e}^{\lambda(y-z)} {\rm d} y.
\end{align}
Note that
$
 {\rm e}^{\lambda(y-z)}G(t,z-y)
 ={\rm e}^{\lambda^2t}G\bigl(t,z-y+2\lambda t\bigr),
$
which with \eqref{eq:h-finite}, we have
\begin{align}\label{eq:h-bound}
 h_f(t,z)
 \le {\rm e}^{\lambda^2t}\|f\|_{L^\infty(\mathbb{R})}
 \int_{\mathbb{R}}G\bigl(t,z-y+2\lambda t\bigr){\rm d} y
 ={\rm e}^{\lambda^2t}\|f\|_{L^\infty(\mathbb{R})}.
\end{align}
This proves \eqref{uni-bound}.
Next, we choose a nonnegative cutoffs
$\chi_N\in C_c^\infty(\mathbb{R})$ such that $0\le\chi_N\le1$ and
$\chi_N\uparrow1$, and set $f_N=\chi_Nf$. Then
$$
 F_N(t,z)=\int_{\mathbb{R}}K_m(t,z,y){\rm e}^{\lambda y}f_N(y){\rm d} y,
$$
Since $|{\rm e}^{\lambda y}f_N(y)|<+\infty$, thus, $F_N(t,z)$ solves
$
 (F_N)_t=(F_N)_{zz}-V_m(z)F_N
$
with bounded, compactly supported initial data
and $V_m$ satisfies \eqref{eq:Km-Sm}.
Combining $h_N={\rm e}^{-\lambda z}F_N$ and differentiating gives
\begin{equation}\label{cut-hN}
 (h_N)_t=(h_N)_{zz}+2\lambda(h_N)_z
 +(\lambda^2-V_m(z))h_N.
\end{equation}
The bound \eqref{cut-hN} is uniform in $N$ on every finite
time. Since
$
 \lambda^2-V_m(z)
 =r(z)-(1+m)\phi(z),
$
and then letting $N\to+\infty$ , we obtain \eqref{eq:h-pde}.

Let $v_0$ be the initial datum of \eqref{eq:coma2} satisfying \eqref{eq:notation}.
Now set $f=|v_0|$ and $h=h_{|v_0|}$ for convenience, which satisfies \eqref{eq:h-pde}. By Lemma \ref{lemma3.1-Eq}, substituting this into \eqref{eq:coma2} gives
\begin{align*}
 &h_t-h_{zz}-2\lambda h_z-(r(z)-u-\phi)h\notag\\
 &=\bigl[r(z)-(1+m)\phi-(r(z)-u-\phi)\bigr]h
 =(u-m\phi)h\ge0.
\end{align*}
Since $h(0,z)=|v_0(z)|\ge v_0(z)$, applying the comparison principle with $v$ and $-v$ gives \eqref{lemma-v}.
Multiplying \eqref{eq:h-definition} by ${\rm e}^{\lambda z}$, we obtain the key
estimate \eqref{lemma-W}. This complete the proof.
\end{proof}

\subsection{Decay estimates for $S_m$ and $K_m$}

In order to obtain the decay estimates of the solution $v$ of \eqref{eq:coma2} and the solution $W$ of \eqref{we2}, rather than only a finite-time bound established in Lemma \ref{lem:linear-up}, we will further give more detailed decay estimates for $S_m(t,z)$ and $K_m(t,z,y)$.
Since $\phi(-\infty)=r_-$, then by \eqref{r-TRANS} and \eqref{eq:Km-Sm} obtains
\[
 \lim_{z\to-\infty}V_m(z)=\lambda^2+mr_->\lambda^2,
\]
we may fix $\kappa\in(\lambda^2,\lambda^2+mr_-)$ and then choose $z_1<-1$
such that
\begin{equation}\label{eq:left}
 V_m(z)\ge\kappa,~z\le z_1;~~{\rm and}~~ \eta:=\kappa-\lambda^2>0.
\end{equation}

\begin{lemma}\label{lem:mass-decay}
Under (H1), there exists some constant $C_\kappa>0$ such that
\begin{equation}\label{S-estimate}
 S_m(t,z)\le C_\kappa
       \min\!\left\{1,\frac{a(z)}{\sqrt t}\right\},
 \quad t>0,~~ z\in\mathbb{R},
\end{equation}
with $a(z)$ defined by \eqref{eq:weight}.
\end{lemma}

\begin{proof}
Because $V_m\ge0$ and \eqref{eq:left} holds, then
$
 V_m(z)\ge\kappa\mathbf 1_{(-\infty,z_1]}(z).
$
Let $S_*(t,z)$ solve
\begin{equation}\label{eq:step-pde}
 \begin{cases}
 (S_*)_t=(S_*)_{zz}
 -\kappa\mathbf 1_{(-\infty,0]}S_*,
 &t>0,\ z\in\mathbb{R},\\
 S_*(0,z)=1,&z\in\mathbb{R}.
 \end{cases}
\end{equation}
By \eqref{kernel-comparison} and  the comparison principle, it is easy to check that
\begin{equation}\label{eq:compar}
 0\le S_m(t,z)\le S_*(t,z-z_1)\le1.
\end{equation}
Define Laplace transform in time by
\begin{equation*}
 R_s(z)=\int_0^{+\infty} {\rm e}^{-s t}S_*(t,z) {\rm d} t.
\end{equation*}
Since $0\le S_*\le1$, the integral is finite and
$0\le R_s\le1/\lambda$. Multiplying
\eqref{eq:step-pde} by ${\rm e}^{-\lambda t}$ and integrating in $t$ gives
\begin{equation*}
 -R_\lambda''
 +\bigl(s+\kappa\mathbf1_{(-\infty,0]}\bigr)R_s=1,
\end{equation*}
which implies that
$$
 -R_s''+(s+\kappa)R_\lambda=1,~~z\leq0;~~{\rm and}~~-R_s''+s R_s=1,~~z>0.
$$
Note that $R_s''\in L^\infty(\mathbb{R})$ in the
distributional sense. Hence
$R_s\in W^{2,\infty}_{\mathrm{loc}}(\mathbb{R})$ and has a $C^1$
regularity. Then the positive bounded solutions have the form
\begin{equation}\label{laplace-wise}
 R_s(z)=
 \begin{cases}
 \displaystyle
 \frac{1}{s+\kappa}
 +B_s {\rm e}^{\sqrt{s+\kappa}\,z},&z\le0,\\[2mm]
 \displaystyle
 \frac{1}{s}
 +A_s {\rm e}^{-\sqrt s\,z},&z\ge0.
 \end{cases}
\end{equation}
Both $R_s$ and $R_s'$ are continuous at $z=0$.
Direct computation yields
\begin{equation*}
 B_s
 =\frac{1}{\sqrt s\sqrt{s+\kappa}}-\frac{1}{s+\kappa}~~{\rm and}~~
 A_s=\frac{1}{\sqrt s\sqrt{s+\kappa}}-\frac{1}{s}.
\end{equation*}
Derivative matching gives
$
 \sqrt{s+\kappa}B_\lambda=-\sqrt s A_\lambda.
$
For $z\le0$, the exponential coefficient in
\eqref{laplace-wise} is nonnegative. Thus,
\begin{align}\label{left-bound}
 s R_s(z)
 &\le s R_s(0)
 =\frac{\sqrt s}{\sqrt{s+\kappa}}
 \le\frac{\sqrt s}{\sqrt\kappa}.
\end{align}
For $z\ge0$,
\begin{align}\label{right-exact}
 s R_s(z)
 &=1-{\rm e}^{-\sqrt s z}
 +\frac{\sqrt s}{\sqrt{s+\kappa}}
 {\rm e}^{-\sqrt s z}.
\end{align}
Using
$
 1-{\rm e}^{-\sqrt s z}\le \sqrt s z
$ for $z\geq 0$ and
 ${\rm e}^{-\sqrt s z}\le1,
$
which with \eqref{right-exact} leads to
\begin{align}\label{right-bound}
 s R_s(z)
 &\le\sqrt s\,z+\frac{\sqrt s}{\sqrt\kappa}
 \le C_\kappa\sqrt s(1+z).
\end{align}
Combining \eqref{left-bound}-\eqref{right-bound} and the definition of $a(z)$ defined by \eqref{eq:weight}, we obtain
\begin{equation}\label{eqnified}
 s R_s(z)
 \le C_\kappa\sqrt s\,(1+z_+),~~ z\in\mathbb{R}.
\end{equation}
Now fix $t>0$ and choose $s=1/t$. Since $S_*(\tau,z)$ decreases in
$\tau$,
\begin{align}\label{eq:laplace-to-time}
 R_{1/t}(z)
 =\int_0^{+\infty} {\rm e}^{-\tau/t}S_*(\tau,z){\rm d} \tau
 \ge S_*(t,z)\int_0^t {\rm e}^{-\tau/t}{\rm d} \tau=t(1-{\rm e}^{-1})S_*(t,z).
\end{align}
Hence, by \eqref{eqnified}, we have
\begin{align*}
 S_*(t,z)
 \le \frac{1}{1-{\rm e}^{-1}}\frac{1}{t}R_{1/t}(z)
 \le C_\kappa\frac{1+z_+}{\sqrt t}.
\end{align*}
Together with $S_*\le1$, we obtain
$
 S_*(t,z)
 \le C_\kappa
 \min\left\{1,\frac{1+z_+}{\sqrt t}\right\},
$
which with $a(z)=1+(z-z_1)_+$ and \eqref{eq:compar} prove \eqref{S-estimate}.
\end{proof}

\begin{lemma}\label{lem-kernel2}
Under (H1), There exists $C>0$ such that
\begin{equation*}
 K_m(t,z,y)
 \le Ct^{-1/2}
 \min\left\{1,\frac{a(z)}{\sqrt t}\right\}
 \min\left\{1,\frac{a(y)}{\sqrt t}\right\},
 \qquad t>0,
\end{equation*}
for all $z,y\in\mathbb{R}$.
\end{lemma}

\begin{proof}
It follows from \eqref{eq:symmetry} that
\begin{align}\label{-identity}
 K_m(t,z,y)
 =\int_{\mathbb{R}}\int_{\mathbb{R}}
 K_m(t/3,z,\xi)
 K_m(t/3,\xi,\zeta)
 K_m(t/3,\zeta,y){\rm d}\xi{\rm d}\zeta.
\end{align}
Using the fact that \eqref{eq:f-bound}, we have pointwise estimate
\begin{equation*}
 K_m(t/3,\xi,\zeta)
 \le G(t/3,\xi-\zeta)
 \le\frac{1}{\sqrt{4\pi(t/3)}}
 =\sqrt{\frac{3}{4\pi}}\,t^{-\frac12}.
\end{equation*}
Substituting this into \eqref{-identity}, and using the
symmetry of $K_m$ and \eqref{eq:Km-Sm} gives
\begin{align*}
 K_m(t,z,y)
 &\le\sqrt{\frac{3}{4\pi}}\,t^{-\frac 12}
 \left(\int_{\mathbb{R}}K_m(t/3,z,\xi){\rm d}\xi\right)
 \left(\int_{\mathbb{R}}K_m(t/3,\zeta,y){\rm d}\zeta\right)\notag\\
 &=\sqrt{\frac{3}{4\pi}}\,t^{-\frac 12}
 S_m(t/3,z)S_m(t/3,y).
\end{align*}
By Lemma~\ref{lem:mass-decay}, we obtain
$
 S_m(t/3,z)
 \le C_1\min\left\{1,\frac{a(z)}{\sqrt{t/3}}\right\}
 \le C_2\min\left\{1,\frac{a(z)}{\sqrt t}\right\}
$
for some constants $C_i>0(i=1,2)$. The same estimate holds for $y$.
Thus, we complete the proof of Lemma \ref{lem-kernel2}.
\end{proof}

In what follows, we also need to prove the decay rate to the left half-line.
\begin{lemma}\label{lem:left-transfer}
Under (H1),
let $f\in{\rm BUC}(\mathbb{R})$ be nonnegative and let $h_f$ be as in
Lemma~\ref{lem:linear-up}. Suppose that $h_f(t,z_1)\le Bt^{-\beta}$ for $t\ge1$ with some $\beta>0$ and
$B\ge0$, then
\begin{equation}\label{left-decay}
 \sup_{z\le z_1}h_f(t,z)
 \le C_{\beta,\eta}\left(\|f\|_{L^\infty(\mathbb{R})}+B\right)t^{-\beta},\quad t\ge1.
\end{equation}
\end{lemma}

\begin{proof}
On $z\le z_1$, \eqref{eq:Km-Sm} and \eqref{eq:left} give
\begin{equation}\label{left-coeff}
 r(z)-(1+m)\phi(z)
 =\lambda^2-V_m(z)
 \le-(\kappa-\lambda^2)=-\eta.
\end{equation}
Set
$
 E={\rm e}^{\lambda^2}\|f\|_{L^\infty(\mathbb{R})}
$ and
$ T_\beta=\max\left\{1,\frac{\beta}{\eta}\right\},$ and define
\begin{equation*}
 F_\beta(t)
 =T_\beta^\beta(E+B)(t-1+T_\beta)^{-\beta},\quad t\ge1,
\end{equation*}
which has
\begin{align*}
 F_\beta'(t)+\eta F_\beta(t)
 &=\left(\eta-
 \frac{\beta}{t-1+T_\beta}\right)F_\beta(t)\ge0.
\end{align*}
Since $F_\beta$ is independent of $z$, by \eqref{eq:h-pde} and \eqref{left-coeff}, we have
\begin{align*}
 &\partial_tF_\beta-\partial_{zz}F_\beta
 -2\lambda\partial_zF_\beta
 -[r(z)-(1+m)\phi(z)]F_\beta \ge F_\beta'(t)+\eta F_\beta(t)\ge0.
\end{align*}
At $t=1$, by \eqref{uni-bound} and $h_f(1,z_1)\le B$ with some $B\geq0$, we obtain
\begin{equation}\label{incom-1}
 h_f(1,z)+h_f(1,z_1)\le {\rm e}^{\lambda^2}\|f\|_{L^\infty(\mathbb{R})}+B=E+B\le F_\beta(1).
\end{equation}
%Next, fix $T>1$ and further define
%$
% \Psi(t,z):={\rm e}^{(t-1)}\bigl(1+(z-z_1)^2\bigr).
%$
%In virtue of \eqref{left-coeff}, we obtain
%\begin{equation*}
% \begin{aligned}
% &\Big(\partial_t-\partial_{zz}-2\lambda\partial_z-[r(z)-(1+m)\phi(z)]\Big)\Psi\\
% \geq&{\rm e}^{2(t-1)}
%       \Big((1+\eta)\bigr(1+(z-z_1)^2\bigr)-1-2\lambda (z-z_1)\Big)\ge0,
% \end{aligned}
%\end{equation*}
%the last inequality holds for $|z-z_1|\gg1$.
%For any $\varepsilon>0$, the function $F_\beta+\varepsilon\Psi$
%is a supersolution on $[1,T]\times[z_1-R,z_1]$ for some $R>0$. It dominates $h_f$
%on the initial and right boundaries. Since $h_f$ is bounded on this
%time interval and $\varepsilon\Psi(t,z_1-R)\ge\varepsilon(1+R^2)$,
%it also dominates on the left boundary when $R$ is sufficiently large.
%Comparison, followed by $R\to\infty$ and $\varepsilon\downarrow0$,
%gives $h_f\le F_\beta$ on $[1,T]\times(-\infty,z_1]$. The time $T$
%is arbitrary.
%
Furthermore,
$
 T_\beta t\ge t-1+T_\beta
$
for $t\ge1$, then
\begin{align*}
 F_\beta(t)
 &\ge(E+B)t^{-\beta}
 \ge Bt^{-\beta}
 \ge h_f(t,z_1).
\end{align*}
Then by comparison principle, the proof is completed.
\end{proof}

\begin{proof}[Proof of Theorem~\ref{thm:stab}]
The global existence and bounds of $u(t,z)$ for \eqref{eq:initi865}
are proved in Lemma~\ref{lemma3.1-Eq}. It remains
to prove the three asymptotic stability assertions.

\medskip
\noindent
\textbf{Proof of \textup{(i)}.}
Since $W_0\in L^\infty(\mathbb{R})$ and set $h=h_{|v_0|}$. By Lemma~\ref{lem:linear-up}, it follows from
\eqref{eq:h-definition} that
\begin{align*}
{\rm e}^{\lambda z}h(t,z)
 &=\int_{\mathbb{R}}K_m(t,z,y)|W_0(y)|{\rm d} y\\
 &\le\|W_0\|_{L^\infty(\mathbb{R})}
 \int_{\mathbb{R}}K_m(t,z,y){\rm d} y
 =\|W_0\|_{L^\infty(\mathbb{R})} S_m(t,z).
\end{align*}
Further by Lemma~\ref{lem:mass-decay}, we have
\begin{equation}\label{eq:i-weighted}
 {\rm e}^{\lambda z}h(t,z)
 \le C\|W_0\|_{L^\infty(\mathbb{R})}
 \min\left\{1,\frac{a(z)}{\sqrt t}\right\}.
\end{equation}
Let $x=z-z_1\ge0$, note that
\begin{equation}\label{ine-bound}
 {\rm e}^{-\lambda z}a(z)
% =e^{-\lambda z_1}(1+x)e^{-\lambda x}
 \le {\rm e}^{-\lambda z_1}
 \left(1+\sup_{x\ge0}x{\rm e}^{-\lambda x}\right)
 ={\rm e}^{-\lambda z_1}\left(1+\frac{1}{{\rm e}\lambda}\right)
<+\infty.
\end{equation}
For $z\ge z_1$, using the fact that
$
 \min\left\{1,\frac{a(z)}{\sqrt t}\right\}
 \le\frac{a(z)}{\sqrt t}
$ and \eqref{eq:i-weighted}, we obtain
\begin{equation}\label{eq:i-right}
 \sup_{z\ge z_1}h(t,z)\le C\|W_0\|_{L^\infty(\mathbb{R})} t^{-\frac12}{\rm e}^{-\lambda z}a(z)
 \le C\|W_0\|_{L^\infty(\mathbb{R})} t^{-\frac12},~~ t\ge1.
\end{equation}
In particular,
$
 h(t,z_1)\le C\|W_0\|_{L^\infty(\mathbb{R})} t^{-1/2}.
$
Applying Lemma~\ref{lem:left-transfer} with
$f=|v_0|$, $\beta=1/2$ and $B=C\|W_0\|_{L^\infty(\mathbb{R})}$ gives
\begin{equation}\label{eq:i-left}
 \sup_{z\le z_1}h(t,z)
 \le C\bigl(||v_0||_{L^\infty(\mathbb{R})}+\|W_0\|_{L^\infty(\mathbb{R})}\bigr)t^{-\frac12},
 \,\,\, t\ge1.
\end{equation}
Since $|v|\le h$, \eqref{eq:i-right} and \eqref{eq:i-left} yield
$$
 \|v(t,\cdot)\|_{L^\infty([z_1,+\infty))}
 \le C\bigl(||v_0||_{L^\infty(\mathbb{R})}+\|W_0\|_{L^\infty(\mathbb{R})}\bigr)t^{-\frac12},\,\,\, t\ge1.
$$
For $0\le t\le1$, use \eqref{uni-bound} and enlarge the constant.
This proves the statement (i).

\medskip
\noindent
\textbf{Proof of \textup{(ii)}.}
By the definition of $M_\sigma$ in \eqref{eq:weight},
assume $0\le\sigma\le1$ and $M_\sigma<+\infty$, and let $h=h_{|v_0|}$. Since $a(z)\ge1$, this
implies $W_0\in L^1(\mathbb{R})$. By Lemmas~\ref{lem:linear-up} and \ref{lem-kernel2},
\begin{align}\label{eq:-bound3}
 {\rm e}^{\lambda z}h(t,z)
 &\le Ct^{-\frac12}
 \min\left\{1,\frac{a(z)}{\sqrt t}\right\}
 \int_{\mathbb{R}}
 \min\left\{1,\frac{a(y)}{\sqrt t}\right\}|W_0(y)| {\rm d} y.
\end{align}
Note that
\begin{align}\label{eq:-factor1}
 \frac{1}{a(z)}
 \min\left\{1,\frac{a(z)}{\sqrt t}\right\}
 &=\min\left\{\frac{1}{a(z)},\frac{1}{\sqrt t}\right\}
 \le t^{-\frac{1}{2}},~~t\geq1.
\end{align}
Further, for every $s>0$ and every $0\le\sigma\le1$, one has
$
 \min\{1,s\}\le s^\sigma.
$
Indeed, if $0<s\le1$, then $s\le s^\sigma$, whereas if $s\ge1$,
then $1\le s^\sigma$. Taking $s=a(y)/\sqrt t$ gives
\begin{equation}\label{eq:data-factor}
 \min\left\{1,\frac{a(y)}{\sqrt t}\right\}
 \le t^{-\frac{\sigma}{2}}\big(a(y)\big)^\sigma,~~t\geq1.
\end{equation}
Therefore
\begin{align}\label{eq:-integral3}
 \int_{\mathbb{R}}
 \min\left\{1,\frac{a(y)}{\sqrt t}\right\}|W_0(y)| {\rm d} y
 \le t^{-\frac{\sigma}{2}}
 \int_{\mathbb{R}}\big(a(y)\big)^\sigma|W_0(y)| {\rm d} y
 =M_\sigma t^{-\frac{\sigma}{2}},~~t\geq1.
\end{align}
Divide \eqref{eq:-bound3} by $a(z)$ and combine
\eqref{eq:-factor1} with \eqref{eq:-integral3}. This gives
\begin{equation}\label{eq:moment-weighted-sup}
 \sup_{z\in\mathbb{R}}\frac{{\rm e}^{\lambda z}h(t,z)}{a(z)}
 \le CM_\sigma t^{-1-\frac\sigma 2},~~ t\ge1.
\end{equation}
For $z\ge z_1$, multiply \eqref{eq:-bound3} by ${\rm e}^{-\lambda z}a(z)$ and use
\eqref{ine-bound}; thus
\begin{equation}\label{eq:ii-right}
 \sup_{z\ge z_1}h(t,z)
 +\sup_{z\ge z_1}\frac{{\rm e}^{\lambda z}h(t,z)}{a(z)}
 \le CM_\sigma t^{-1-\frac\sigma 2},~~t\geq1.
\end{equation}
In particular, because $a(z_1)=1$,
$
 h(t,z_1)\le CM_\sigma t^{-1-\sigma/2}.
$
Thus, by applying Lemma~\ref{lem:left-transfer} with
$ f=|v_0|,~\beta=1+\frac\sigma2$ and $B=CM_\sigma$.
We obtain
\begin{equation}\label{eq:ii-left}
 \sup_{z\le z_1}h(t,z)
 \le C(||v_0||_{L^\infty(\mathbb{R})}+M_\sigma)t^{-1-\frac\sigma 2},~~t\geq1.
\end{equation}
Finally, since $|v|\le h$, we obtain
\begin{align*}
 \|v(t,\cdot)\|_{L^\infty(\mathbb{R})}+\|v(t,\cdot)\|_X
 &\le
 2\sup_{z\le z_1}h(t,z)
 +\sup_{z\ge z_1}h(t,z)
 +\sup_{z\ge z_1}\frac{{\rm e}^{\lambda z}h(t,z)}{a(z)}\\
 &\le C(||v_0||_{L^\infty(\mathbb{R})}+M_\sigma)t^{-1-\frac\sigma 2},~~t\geq1.
\end{align*}
This proves the statement (ii).

\medskip
\noindent
\textbf{Proof of \textup{(iii)}.}
Assume $\lim\limits_{z\to+\infty}\frac{u_0(z)}{\phi(z)}=1$. Since $v_0=u_0-\phi$ and $\phi>0$, then
$$
 \frac{|v_0(z)|}{\phi(z)}\to0,~~{\rm as}~z\to+\infty.
$$
Fix $\varepsilon>0$, there exists $R_\varepsilon>z_1$ such that
$|v_0(z)|\le\varepsilon\phi(z)$ for $z\ge R_\varepsilon$.
Define
\begin{equation}\label{eq:remainder-definition}
 g_\varepsilon(z)
 =\bigl(|v_0(z)|-\varepsilon\phi(z)\bigr)_+,
\end{equation}
which implies that
$
 0\le g_\varepsilon\le ||v_0||_{L^\infty(\mathbb{R})}$~
and $g_\varepsilon(z)=0$ for $z\ge R_\varepsilon$, and
\begin{equation}\label{eq:remainder-decomposition}
 |v_0(z)|\le\varepsilon\phi(z)+g_\varepsilon(z).
\end{equation}
Moreover,
\begin{align}\label{eq:uperest-1}
 \int_{\mathbb{R}}a(y){\rm e}^{\lambda y}g_\varepsilon(y){\rm d} y
 \le ||v_0||_{L^\infty(\mathbb{R})}
 \left[
 \int_{-\infty}^{z_1}{\rm e}^{\lambda y}{\rm d} y
 +\int_{z_1}^{R_\varepsilon}
 (1+y-z_1){\rm e}^{\lambda y} {\rm d} y
 \right]<+\infty.
\end{align}
Let $h_\varepsilon=h_{g_\varepsilon}$ be the solution defined by
\eqref{eq:h-definition}. Define
$$
 \overline v_\varepsilon(t,z)
 =\varepsilon\phi(z)+h_\varepsilon(t,z).
$$
Using the profile equation \eqref{1.3} and \eqref{eq:h-pde}, we compute
\begin{align*}
 &\partial_t\overline v_\varepsilon
 -\partial_{zz}\overline v_\varepsilon
 -2\lambda\partial_z\overline v_\varepsilon
 -(r(z)-u-\phi)\overline v_\varepsilon\\
 &=\varepsilon u(t,z)\phi(z)
 +\bigl(u(t,z)-m\phi(z)\bigr)h_\varepsilon(t,z)
 \ge0.
\end{align*}
By \eqref{eq:remainder-decomposition} gives
$\overline v_\varepsilon(0,z)\ge\pm v_0(z)$, comparison with $v$ and
$-v$ yields
\begin{equation}\label{eq:tail-comparison}
 |v(t,z)|\le\varepsilon\phi(z)+h_\varepsilon(t,z).
\end{equation}

It remains to show $\|h_\varepsilon(t,\cdot)\|_{L^\infty(\mathbb{R})}\to0$. Thanks to \eqref{eq:uperest-1}, by Lemmas~\ref{lem:linear-up} and
\ref{lem-kernel2}, which applied with $\sigma=1$ to the initial function
$g_\varepsilon$, this leads to
$$
 \sup\limits_{z\in\mathbb{R}}
 \frac{{\rm e}^{\lambda z}h_\varepsilon(t,z)}{a(z)}
 \le CN_\varepsilon t^{-\frac32},~~t\geq1.
$$
Hence, by \eqref{eq:i-weighted} and \eqref{eq:ii-right}, then
$
 \sup_{z\ge z_1}h_\varepsilon(t,z)
 \le CN_\varepsilon t^{-3/2}
 $ for $t\geq1$.
Applying Lemma~\ref{lem:left-transfer} with
$
 f=g_\varepsilon,~\beta=3/2$, and $B=CN_\varepsilon$, we obtain
\begin{equation}\label{eq:remainder-left}
 \sup_{z\le z_1}h_\varepsilon(t,z)
 \le C\bigl(\|g_\varepsilon\|_{L^\infty(\mathbb{R})}+N_\varepsilon\bigr)t^{-\frac32}.
\end{equation}
Thus, we have $\|h_\varepsilon(t,\cdot)\|_{L^\infty(\mathbb{R})}\to0$ as $t\to+\infty$.
Finally, from \eqref{eq:tail-comparison},
$$
 \limsup_{t\to\infty}\|u(t,\cdot)-\phi(\cdot)\|_{L^\infty(\mathbb{R})}
 \le\varepsilon\|\phi(\cdot)\|_{L^\infty(\mathbb{R})}.
$$
by the arbitrariness of $\varepsilon>0$, thus,
$
 \lim\limits_{t\to\infty}\|u(t,\cdot)-\phi(\cdot)\|_{L^\infty(\mathbb{R})}=0.
$
This proves \textup{(iii)} and completes the proof.
\end{proof}

\section*{Appendix A}

{\bf Proof of Lemma \ref{lemma1}.}
It follows from \eqref{qeqlv} that
$
zq(z)=z^{1-p}L(z).
$
Since $1-p<-1$, Karamata's theorem \cite[Theorem 1.5.11]{BGT1987} implies that
$
\int_{z_0}^\infty s q(s)\,{\rm d}s<+\infty
$
for every sufficiently large $z_0$.
We first construct a solution which converges to $1$ at $+\infty$. Choose $z_0>0$ sufficiently large such that
$
\theta_{z_0}:=\int_{z_0}^{+\infty} s q(s)\,{\rm d}s<1.
$
Define
$$
(\mathrm{T} v)(z)
=
1+\int_z^{+\infty} (s-z)q(s)v(s)\,{\rm d}s.
$$
For any $v,w\in C([z_0,+\infty))$,
$$
\begin{aligned}
|(\mathrm{T}v)(z)-(\mathrm{T}w)(z)|
&\leq
\int_z^{+\infty} (s-z)q(s)|v(s)-w(s)|\,{\rm d}s \\
&\leq
\|v-w\|_{L^\infty([z_0,+\infty))}
\int_z^{+\infty} s q(s)\,{\rm d}s
\leq
\theta_{z_0}\|v-w\|_{L^\infty([z_0,+\infty))}.
\end{aligned}
$$
Since $\theta_{z_0}<1$, the operator $\mathrm{T}$ is a contraction. Therefore, by the Banach fixed-point theorem, there exists a unique bounded function $v_-\in C([z_0,+\infty))$ satisfying
$$
v_-(z)
=
1+\int_z^{+\infty} (s-z)q(s)v_-(s)\,{\rm d}s,
$$
which implies that
$
|v_-(z)-1|
\leq
C\int_z^{+\infty} s q(s)\,{\rm d}s
$
for some constant $C>0$.
Similarly,
$$
|v_-'(z)|
=
\Big|-\int_z^{+\infty} q(s)v_-(s)\,{\rm d}s\Big|
\leq
C\int_z^{+\infty} q(s)\,{\rm d}s,
$$
and further
by Karamata's theorem \cite[Theorem 1.5.11]{BGT1987} for $p>2$,
\begin{equation}\label{KA}
\int_z^{+\infty} sq(s)\,{\rm d}s
\sim
\frac{z^{2-p}L(z)}{p-2}~~{\rm and}~~
\int_z^{+\infty} q(s)\,{\rm d}s
\sim
\frac{z^{1-p}L(z)}{p-1},~z\to+\infty,
\end{equation}
we have
$$
v_-(z) = 1 + O\Big(z^{2-p}L(z)\Big),~~
v_-'(z)
=
O\!\left(z^{1-p}L(z)\right).
$$

To construct a second linearly independent solution, by using the standard reduction-of-order method, we have
\begin{equation}\label{EQ:v_+}
v_+(z)
=
v_-(z)\int_{z_0}^z\frac{{\rm d}s}{v_-^2(s)}.
\end{equation}
Since $v_-(z)\to 1$, then $\frac{1}{v_-^2(z)}\to 1$.
Consequently,
$
\frac{1}{z}
\int_{z_0}^z\frac{{\rm d}s}{v_-^2(s)}
\to 1~{\rm as}~z\to+\infty.
$
Hence
$$
v_+(z)
=
v_-(z)\left(z+o(z)\right)
=
z+o(z),~{\rm as}~z\to+\infty.
$$
Differentiating \eqref{EQ:v_+} with respects to $v_+$ gives
\begin{equation}\label{v_+-EQ}
v_+'(z)
=
v_-'(z)\int_{z_0}^z\frac{{\rm d}s}{v_-^2(s)}
+
\frac{1}{v_-(z)}.
\end{equation}
Note that $zv_-'(z)\to0$.
Indeed,
$
|zv_-'(z)|
\leq
Cz\int_z^{+\infty} q(s)\,{\rm d}s,
$
and by \eqref{KA} gives
$$
z\int_z^{+\infty} q(s)\,{\rm d}s
\sim
\frac{z^{2-p}L(z)}{p-1}
\to0,~{\rm as}~z\to+\infty,
$$
because $p>2$, which with \eqref{v_+-EQ} obtains that
$
v_+'(z)=1+o(1).
$
Finally, by direct calculation,
$$
\begin{aligned}
W(v_-,v_+)
&=
-\Big(v_-(z)v_+'(z)-v_-'(z)v_+(z)\Big) \\
&=
-v_-(z)
\left[
v_-'(z)\int_{z_0}^z\frac{{\rm d}s}{v_-^2(s)}
+
\frac{1}{v_-(z)}
\right]
+
v_-'(z)v_-(z)
\int_{z_0}^z\frac{{\rm d}s}{v_-^2(s)}
=-1.
\end{aligned}
$$
Therefore $v_-$ and $v_+$ are linearly independent and form fundamental solutions.
This completes the proof.

\section*{Appendix B}

{\bf Proof of Lemma \ref{lem:LG-RV}.}
The proof is based on a first-order term of the
Liouville--Green reduction and the classical asymptotic-integration argument of Hartman--Wintner type. The main point is that only the first
derivative of $q(z)$ is required.

\textbf{Step 1.}
We first give the estimates on $q(z)$
which will control the the two fundamental solutions $(v_-(z),v_+(z))$.
Thanks to \eqref{q(z)-Eq}, we have
$$
\frac{q'(z)}{q(z)}
=
-\frac{p}{z}+\frac{L'(z)}{L(z)}.
$$
By the assumption of \eqref{L(z)-Eq} satisfying $zL'(z)/L(z)\to 0$ as $z\to+\infty$, it follows that
$$
q'(z)\big(q(z)\big)^{-\frac32}
=
\frac{-p} {z\sqrt{q(z)}}(1+o(1))~~{\rm and}~~z\sqrt{q(z)}=z^{1-\frac p2}\big(L(z)\big)^{\frac 12}\to+\infty,
$$
we obtain
$$
q'(z)\big(q(z)\big)^{-\frac32}\to0 \,\,\,\,{\rm and}\,\,\,\,
(q'(z))^2\big(q(z)\big)^{-\frac52}
=
O\left(
z^{-2+\frac p2}\big(L(z)\big)^{-\frac12}
\right).
$$
The function $L^{-1/2}$ is slowly varying, while
$
-2+\frac p2<-1.
$
Therefore, Karamata's theorem \cite[Theorem 1.5.11]{BGT1987} yields
\begin{equation}\label{EQ-inte1}
\int_{z_0}^{+\infty}
\frac{(q'(z))^2}{q(z)^{5/2}}\, {\rm d}z
<+\infty\,\,\,\,{\rm and}\,\,\,\,
S(z)
\sim
\frac{z^{1-\frac p2}\big(L(z)\big)^{\frac12}}{1-p/2},
\end{equation}
where $S(z)=\int_{z_0}^{z}\sqrt{q(s)}\, {\rm d}s$ is defined in \eqref{phase}.
In particular, $S(z)\to+\infty$ as $z\to+\infty$.
Further, $\pm\sqrt{q(z)}$ are two eigenvalues of \eqref{linear-EQ}.
It is natural to introduce the Liouville phase variable $t=S(z)$, then
\begin{equation}\label{sqrtqz}
\frac{{\rm d}t}{{\rm d}z}=\sqrt{q(z)}>0,
\end{equation}
so that $z\mapsto t$ is strictly increasing. The usual Liouville--Green amplitude is $q^{-1/4}$. Motivated by
the scaling (see, \cite[Part III]{HW1955}), for a solution $v(z)$ of \eqref{linear-EQ} we introduce
the variables
\begin{equation}\label{sqrtuw-EQ}
u(t)
:=
\big(q(z(t))\big)^{\frac14}v(z(t)),
\quad
w(t)
:=
\big(q(z(t))\big)^{-\frac14}v'(z(t)),~~v'(z)=\frac{{\rm d}v}{{\rm d}z}.
\end{equation}
%This choice is made so that the leading part of the resulting
%first-order system is constant and its two eigenvalues are
%$+1$ and $-1$, corresponding respectively to the growing and
%decaying Liouville--Green modes.
A direct calculation gives
\begin{equation}\label{trans-uv}
u_t=a(t)u+w,
\qquad
w_t=u-a(t)w,
\end{equation}
where $a(t):=\frac 14q'(z(t))\big(q(z(t))\big)^{-\frac32}$.
Moreover, by \eqref{EQ-inte1} and \eqref{sqrtqz}, we obtain
$$
\begin{aligned}
\int_{t_0}^{+\infty}a^2(t)\, {\rm d}t
&=
\frac1{16}
\int_{z_0}^{+\infty}
\frac{(q'(z))^2}{\big(q(z)\big)^{5/2}}\, {\rm d}z
<+\infty.
\end{aligned}
$$
Hence $a(\cdot)\in L^2([t_0,+\infty))$.

{\bf Step 2.}
We shall construct the growing mode. Set
\begin{equation}\label{TRAN-uwxi}
\xi
=
\frac{u+w}{\sqrt2}
\qquad{\rm and}\qquad
\eta
=
\frac{u-w}{\sqrt2}.
\end{equation}
Then \eqref{trans-uv} becomes
\begin{equation}\label{xieta}
\xi_t=\xi+a(t)\eta\,\,~~{\rm and}\,\,\,\,
\eta_t=-\eta+a(t)\xi.
\end{equation}
Note that the variables $\xi$ and $\eta$ correspond to
the growing and decaying modes ${\rm e}^t$ and ${\rm e}^{-t}$, respectively,
we introduce the ratio
\begin{equation}\label{ratio}
\gamma(t):=\frac{\eta(t)}{\xi(t)},
\end{equation}
which with \eqref{xieta}, we obtain
\begin{equation}\label{xieta2}
\gamma'=-2\gamma+a(t)(1-\gamma^2).
\end{equation}
Choose $t_0>0$ sufficiently large and consider the solution of
\eqref{xieta2} satisfying $\gamma(t_0)=0$.
By variation of constants,
$$
|\gamma(t)|
=
\left|\int_{t_0}^t
{\rm e}^{-2(t-s)}
a(s)\big(1-\gamma^2(s)\big)\, {\rm d}s\right|\le
\int_{t_0}^t {\rm e}^{-2(t-s)}|a(s)|\, {\rm d}s.
$$
Since $a(t)\to0$ as $t\to+\infty$, it follows that
$
\gamma(t)\to0
\,\,\text{as }t\to+\infty.
$
This together with $a(\cdot)\in L^2([t_0,+\infty))$ and by
Young's inequality, yields $\gamma(\cdot)\in L^2([t_0,+\infty))$.
Thus, by the Cauchy--Schwarz inequality,
$
a(\cdot)\gamma(\cdot)\in L^1([t_0,+\infty)).
$

By \eqref{ratio}, the first equation of the system \eqref{xieta} becomes
$
\xi_t=(1+a(t)\gamma(t))\xi.
$
Hence
$$
\xi(t)
=
\xi(t_0)
\exp\left(
\big(t-t_0\big)+\int_{t_0}^t a(s)\gamma(s)\, {\rm d}s
\right).
$$
By some suitable normalization, we obtain
$
\xi(t)=\sqrt2\, {\rm e}^t(1+o(1))$ and
$\eta(t)=o({\rm e}^t)$,
which by \eqref{sqrtuw-EQ} and \eqref{TRAN-uwxi}, it follows that
\begin{equation}\label{EQ_v_+}
v_+(z)
=
\big(q(z)\big)^{-\frac14}e^{S(z)}(1+o(1))\,\,\,{\rm and}\,\,\,
\frac{v_+'(z)}{v_+(z)}
=
\sqrt{q(z)}(1+o(1)).
\end{equation}

The second solution can be constructed by
the standard reduction-of-order formula. Thus, we define
$$
\widetilde v_-(z)
:=
v_+(z)
\int_z^\infty\frac{{\rm d}s}{v^2_+(s)}.
$$
The integral is finite for sufficiently large $z$ by
\eqref{EQ_v_+}, and $\widetilde v_-$ is linearly independent of
$v_+(z)$. Further, note that
$$
\widetilde v_-'(z)
=
v_+'(z)
\int_z^\infty\frac{ds}{v^2_+(s)}
-\frac1{v_+(z)}.
$$
By the similar argument as in Lemma \ref{lemma1}, it is easy to obtain that
$$
v_-(z):=2\widetilde v_-(z)=
\big(q(z)\big)^{-\frac14}{\rm e}^{-S(z)}(1+o(1))\,\,\,{\rm and}\,\,\,\frac{v_-'(z)}{v_-(z)}
=
-\sqrt{q(z)}(1+o(1)),~{\rm as}~z\to+\infty,
$$
where $S(z)=\int_{z_0}^{z}\sqrt{q(s)}\, {\rm d}s$.
Finally, since $v_-=2\widetilde v_-$, by direct calculation,
$
v_+(z)v_-'(z)-v_+'(z)v_-(z)
\equiv-2.
$
The proof is completed.

\section*{Appendix C}

{\bf Proof of Lemma \ref{lem:3.1}.}
The main point is that only the first
derivative of $L(z)$ is required.
We first give the classical logarithmic transformation, which converts \eqref{eq:p2-linear} into a standard second-order equation on the logarithmic scale.
Set $$t=\ln z,~~~U(t)={\rm e}^{-\frac{1}{2}t}v({\rm e}^t),$$
then \eqref{eq:p2-linear} is equivalent to
\begin{equation}\label{eq:p2-U}
    U''(t)-\rho^2(t)U(t)=0,
\end{equation}
where
$\rho(t):=\frac{\sqrt{1+4L({\rm e}^t)}}{2}.
$
In particular, $\rho(t)\ge\frac12$.
We divide the proof into several steps.

\textbf{Step 1.}
We observe that \eqref{eq:p2-U} is equivalent to equation\eqref{eq:p2-linear}. Thus, it suffices to analyze \eqref{eq:p2-U}.
We first provide the relevant error estimates.
Since $\rho(t)=\frac{\sqrt{1+4L({\rm e}^t)}}{2}$,
by $z={\rm e}^t$ gives
\begin{equation*}
   \left|\frac{\rho'(t)}{\rho(t)}\right| = \left|\frac{zL'(z)}
   {2\left(1/4+L(z)\right)}\right|  \le \frac12\left|\frac{zL'(z)}{L(z)}\right|.
\end{equation*}
Therefore, $\rho'(t)/\rho(t)\to 0$ as $t\to+\infty$.
We next verify the error condition.
Since ${\rm d}t={\rm d}z/z$, we have
\begin{equation}\label{estimate-EQ}
    \int_{{\rm e}^{t_0}}^{+\infty}\frac{(\rho'(t))^2}{\rho^3(t)}\, {\rm d}t
    = \frac14 \int_{z_0}^{+\infty}\frac{ z(L'(z))^2 }{ \left(1/4+L(z)\right)^{5/2}}\, {\rm d}z.
\end{equation}

Let
$
\varepsilon(z):=\frac{zL'(z)}{L(z)}.
$
Then $\varepsilon(z)\to0$ as $z\to+\infty$, and hence $\varepsilon$ is bounded on $z\ge z_0$. We shall
divide two cases:

Case (i): suppose that $L(z)$ is non-increasing. Then
$L'(z)\le0$ and, by monotonicity,
$ L(z)\to\ell$ for some $\ell\ge0$.
Since
$$
z\big(L'(z)\big)^2=|\varepsilon(z)|\,L(z)|L'(z)|,~~{\rm and}~~\Big(\frac14+L(z)\Big)^{-\frac 52}\le C,
$$
for some constant $C>0$,
then there exists sufficiently large $z_0>0$ such that

\begin{align}\label{eq-y}
    \int_{z_0}^{+\infty}
    \frac{z(L')^2}
    {(\frac14+L)^{5/2}}\, {\rm d}z
    &\le C\sup_{z\ge z_0}|\varepsilon(z)|\int_{z_0}^\infty L(z)|L'(z)|\, {\rm d}z\notag\\
    &=
    \frac{C}{2}\sup_{z\ge z_0}|\varepsilon(z)|\bigl(L^2(z_0)-\ell^2\bigr)<+\infty.
    \end{align}

{Case (ii):}
suppose that $L(z)$ is nondecreasing. If $L$ is bounded,
then $L(z)\to\ell\in(0,\infty)$, and the same argument gives
\eqref{eq-y}. If $L(z)\to+\infty$, then, increasing $z_0$ if necessary, $L(z)\ge1$ for $z\ge z_0$, and
\begin{align*}
    \frac{z(L')^2}
    {(\frac14+L)^{5/2}}=
    \varepsilon(z)
    \frac{L(z)L'(z)}
    {(\frac14+L(z))^{5/2}}\le
    C|\varepsilon(z)|
    \frac{L'(z)}{\big(L(z)\big)^{3/2}}.
\end{align*}
Consequently,
\begin{align*}
    \int_{z_0}^{+\infty}
    \frac{z(L')^2}{(\frac14+L)^{5/2}}\, {\rm d}z
    &\le C\sup_{z\ge z_0}|\varepsilon(z)|  \int_{z_0}^{+\infty}
    \frac{L'(z)}{\big(L(z)\big)^{3/2}}\, {\rm d}z
    = \frac{ 2C\sup_{z\ge z_0}|\varepsilon(z)| }{\sqrt{L(z_0)}}<+\infty.
\end{align*}
Thus \eqref{estimate-EQ} converges in all cases.

\textbf{Step 2.} Set
\begin{equation}\label{define0}
t_0:=\ln z_0,
\quad
\tau(t):=\int_{t_0}^t \rho(s)\, {\rm d}s.
\end{equation}
Note that $\rho(\cdot)\in C^1([t_0,+\infty))$ and
$
\rho(t)\ge\frac12.
$
Hence $\tau(t)$ is strictly increasing and
$
\tau(t)\to+\infty$ as $t\to+\infty$.
Denote the inverse function by $t=t(\tau)$. Define
\[
\kappa(\tau)
:=
\frac{\rho'(t(\tau))}{\rho^2(t(\tau))}.
\]
Since ${\rm d}\tau=a(t)\, {\rm d}t$, it follows from Step 1 that
\begin{equation}\label{int-ine}
\int_0^\infty |\kappa(\tau)|^2\, {\rm d}\tau
=
\int_{t_0}^{+\infty}\frac{(\rho'(t))^2}{\rho^3(t)}\, {\rm d}t<+\infty.
\end{equation}

We now apply the classical Liouville--Green phase--amplitude
normalization to \eqref{eq:p2-U}. Set
$
Q(t):=\rho^2(t),
$
and define
$$
\tau(t)
:=
\int_{t_0}^{t}\sqrt{Q(s)}\,\mathrm ds
=
\int_{t_0}^{t}\rho(s)\,\mathrm ds.
$$
Writing \(t=t(\tau)\), by the same transformation \eqref{sqrtuw-EQ} introduced in Lemma \ref{lem:LG-RV}, we define
\begin{equation}\label{DE:tran-P}
\begin{cases}
\mathcal U(\tau)
:=
\big(Q(t(\tau))\big)^{\frac14}U(t(\tau))
=
\sqrt{\rho(t(\tau))}\,U(t(\tau)),\\
\mathcal W(\tau)
:=
\big(Q(t(\tau))\big)^{-\frac14}U'(t(\tau))
=
\frac{U'(t(\tau))}{\sqrt{\rho(t(\tau))}}.
\end{cases}
\end{equation}
Then let
\begin{equation}\label{EQ:ZZP}
Z_+(\tau)
:=
\frac{\mathcal U(\tau)+\mathcal W(\tau)}{2},
\qquad
Z_-(\tau)
:=
\frac{\mathcal U(\tau)-\mathcal W(\tau)}{2},
\end{equation}
which with \eqref{DE:tran-P} yields that
\begin{equation}\label{EQ:weifen}
\begin{aligned}
\frac{{\rm d}}{{\rm d}\tau} Z_+
&=
\frac12\left(\frac{{\rm d}}{{\rm d}\tau}\mathcal U+\frac{{\rm d}}{{\rm d}\tau}\mathcal W\right)
=
Z_++\kappa(\tau)Z_-,
\\
\frac{{\rm d}}{{\rm d}\tau} Z_-
&=
\frac12\left(\frac{{\rm d}}{{\rm d}\tau}\mathcal U-\frac{{\rm d}}{{\rm d}\tau}\mathcal W\right)
=
\kappa(\tau)Z_+-Z_-,
\end{aligned}
\end{equation}
with $\kappa(\tau):=\frac{\rho'(t(\tau))}{2\rho(t(\tau))^2}.$
Also note that
$$
\frac{\rm dt}{\rm d\tau}
=
\frac{1}{\rho(t(\tau))}
\,\,~~{\rm and}\,\,~~
\frac{\rm d}{\rm d\tau}
\sqrt{\rho(t(\tau))}
=
\kappa(\tau)\sqrt{\rho(t(\tau))},
$$
Thus, setting
$
Z(\tau):=
\left(Z_+(\tau),\,Z_-(\tau)\right)^\top$,
the system \eqref{EQ:weifen} can be written as
\begin{equation}\label{eq:Z-system}
Z'
=
\bigl(J+G(\tau)\bigr)Z,
\end{equation}
where
$$
J:=
\begin{bmatrix}
1&0\\
0&-1
\end{bmatrix},
\qquad
G(\tau)
:=
\kappa(\tau)
\begin{bmatrix}
0&1\\
1&0
\end{bmatrix}.
$$
By \eqref{int-ine}, we have $\kappa(\cdot)\in L^2(0,+\infty)$,  which implies that $|G(\cdot)| \in L^2(0,+\infty)$.
Therefore, the classical
Hartman--Wintner asymptotic integration theorem \cite[Part III]{HW1955}
applies to \eqref{eq:Z-system}, we obtain two
solutions $Z^{(+)}(\tau)$ and $Z^{(-)}(\tau)$, satisfying
\begin{equation}\label{eq:critical-Z-asymptotics}
\begin{aligned}
Z^{(+)}(\tau)
=
{\rm e}^\tau\bigl(1+o(1)\bigr)~~~{\rm and}~~~
Z^{(-)}(\tau)
=
{\rm e}^{-\tau}\bigl(1+o(1)\bigr),
\end{aligned}
\qquad
\tau\to+\infty.
\end{equation}
It follows from that \eqref{DE:tran-P} and \eqref{EQ:ZZP} gives that
$$
U_+(t)
=
\frac{{\rm e}^{\tau(t)}}{\sqrt{\rho(t)}}
\bigl(1+o(1)\bigr),~~~~
U_+'(t)
=
\sqrt{\rho(t)}\,{\rm e}^{\tau(t)}
\bigl(1+o(1)\bigr),
$$
and
$$
U_-(t)
=
\frac{{\rm e}^{-\tau(t)}}{\sqrt{\rho(t)}}
\bigl(1+o(1)\bigr),~~~~
U_-'(t)
=
-\sqrt{\rho(t)}\,{\rm e}^{-\tau(t)}
\bigl(1+o(1)\bigr).
$$
Thus \(U_+\) and \(U_-\) form a fundamental system for
$U''=\rho^2(t)U$ on a sufficiently large interval.
Consequently,
$$
W(U_+,U_-)
:=
U_+U_-'-U_+'U_-
=
-2+o(1)
\qquad\text{as }t\to\infty.
$$
Using the fact that $W'(U_+,U_-)=0$. Therefore
$
W(U_+,U_-)\equiv -2.
$
In particular, \(U_+\) and \(U_-\) are linearly independent and
form fundamental solutions.

%\section*{Acknowledgment}
%Shi-Liang Wu and Zhibao Tang are supported by the NSF of China (No. 12571185),
%the Shaanxi Fundamental Science Research Project for Mathematics and Physics (No. 22JSY029),
%and Natural Science Basic Research Program of Shaanxi (No. 2024RS-CXTD-88).
%Lina Wang is partially supported by .

%%%%%%%%%%%%%%%%%%%%%%%%%%%%%%%%%%%%%%%%%%%%%%%%%%%%%%%%%%%%%%%%%%%%%%%%%%%

\end{document}